\documentclass[a4paper,12pt]{article}
\title{{\bf Multiple cover formula of generalized DT invariants I:
parabolic stable pairs}}
\date{}
\author{Yukinobu Toda}

\usepackage{makeidx}

\usepackage{latexsym}
\usepackage{amscd}
\usepackage{amsmath}
\usepackage{amssymb}
\usepackage{amsthm}
\usepackage{float}
\usepackage[dvips]{graphicx}

\usepackage[all,ps,dvips]{xy}

\usepackage{array}
\usepackage{amscd}
\usepackage[all]{xy}
\usepackage{makeidx}
\usepackage{latexsym}
\DeclareFontFamily{U}{rsfs}{%
\skewchar\font127}
\DeclareFontShape{U}{rsfs}{m}{n}{%
<-6>rsfs5<6-8.5>rsfs7<8.5->rsfs10}{}
\DeclareSymbolFont{rsfs}{U}{rsfs}{m}{n}
\DeclareSymbolFontAlphabet
{\mathrsfs}{rsfs}
\DeclareRobustCommand*\rsfs{%
\@fontswitch\relax\mathrsfs}
\theoremstyle{plain}
\newtheorem{thm}{Theorem}[section]
\newtheorem{prop}[thm]{Proposition}
\newtheorem{lem}[thm]{Lemma}

\newtheorem{defi}[thm]{Definition}
\newtheorem{rmk}[thm]{Remark}
\newtheorem{cor}[thm]{Corollary}

\newtheorem{prop-defi}[thm]{Proposition-Definition}
\newtheorem{thm-defi}[thm]{Theorem-Definition}
\newtheorem{lem-defi}[thm]{Lemma-Definition}

\newtheorem{conj}[thm]{Conjecture}
\newtheorem{exam}[thm]{Example}

\newdimen\argwidth
\def\db[#1\db]{
 \setbox0=\hbox{$#1$}\argwidth=\wd0
 \setbox0=\hbox{$\left[\box0\right]$}
  \advance\argwidth by -\wd0
 \left[\kern.3\argwidth\box0 \kern.3\argwidth\right]}

\newcommand{\aA}{\mathcal{A}}

\newcommand{\eE}{\mathcal{E}}
\newcommand{\fF}{\mathcal{F}}

\newcommand{\hH}{\mathcal{H}}

\newcommand{\lL}{\mathcal{L}}
\newcommand{\mM}{\mathcal{M}}
\newcommand{\nN}{\mathcal{N}}
\newcommand{\oO}{\mathcal{O}}

\newcommand{\Supp}{\mathop{\rm Supp}\nolimits}
\newcommand{\Hom}{\mathop{\rm Hom}\nolimits}

\newcommand{\dR}{\mathbf{R}}

\newcommand{\Pic}{\mathop{\rm Pic}\nolimits}

\newcommand{\Chow}{\mathop{\rm Chow}\nolimits}

\newcommand{\id}{\textrm{id}}

\newcommand{\Ext}{\mathop{\rm Ext}\nolimits}
\newcommand{\Spec}{\mathop{\rm Spec}\nolimits}

\newcommand{\Coh}{\mathop{\rm Coh}\nolimits}

\newcommand{\divv}{\mathop{\rm div}\nolimits}

\newcommand{\cneq}{\mathrel{\raise.095ex\hbox{:}\mkern-4.2mu=}}
\newcommand{\eqcn}{\mathrel{=\mkern-4.5mu\raise.095ex\hbox{:}}}

\newcommand{\Cok}{\mathop{\rm Cok}\nolimits}

\newcommand{\Aut}{\mathop{\rm Aut}\nolimits}

\newcommand{\SL}{\mathop{\rm SL}\nolimits}

\newcommand{\DT}{\mathop{\rm DT}\nolimits}
\newcommand{\PT}{\mathop{\rm PT}\nolimits}

\newcommand{\Imm}{\mathop{\rm Im}\nolimits}

\newcommand{\Ker}{\mathop{\rm Ker}\nolimits}

\newcommand{\GL}{\mathop{\rm GL}\nolimits}

\newcommand{\cl}{\mathop{\rm cl}\nolimits}
\begin{document}
\maketitle
\begin{abstract}
In this paper, we introduce the notion 
of parabolic stable pairs on Calabi-Yau 
3-folds
and invariants counting them. 
By applying the wall-crossing formula
developed by Joyce-Song, Kontsevich-Soibelman, 
 we see that
they are related to 
generalized Donaldson-Thomas invariants 
counting one dimensional semistable sheaves
on Calabi-Yau 3-folds. 
Consequently,  
the conjectural multiple cover formula 
of generalized DT invariants is shown to be 
 equivalent 
to a certain product expansion formula of 
the generating series of parabolic 
stable pair invariants.
The application of this result 
to the multiple cover formula will be 
pursued in the subsequent paper. 
\end{abstract}
\section{Introduction}
The purpose of this paper is to introduce 
the notion of parabolic stable pairs, 
construct their moduli spaces and counting invariants. 
This is a new notion relevant 
to curve
counting invariants on Calabi-Yau 3-folds, 
and very similar to 
classical parabolic vector bundles on curves~\cite{MeSe}, 
coherent systems~\cite{LeP} and stable pairs~\cite{PT}, \cite{JS}. 
The wall-crossing formula established by Joyce-Song~\cite{JS} and
Kontsevich-Soibelman~\cite{K-S} are applied 
to the study of parabolic stable pairs.
Applying the wall-crossing formula, 
we see that 
the parabolic stable pair invariants
are related to generalized Donaldson-Thomas (DT)
invariants counting one dimensional semistable 
sheaves on Calabi-Yau 3-folds. 
On the other hand, the generalized DT invariants 
are expected to satisfy a certain multiple cover
formula, which is equivalent to 
the strong rationality conjecture 
of the generating series of Pandharipande-Thomas (PT)
invariants. 
The conjectural multiple cover formula 
is shown to be equivalent to 
a certain product expansion of the 
generating series of parabolic 
stable pair invariants. 
In a subsequent 
paper~\cite{Todmu}, by applying the results in the present paper, 
we will show 
the multiple cover formula for generalized DT invariants 
in some cases.  

\subsection{Motivation}
Let $X$ be a smooth 
projective Calabi-Yau 3-fold over $\mathbb{C}$, i.e. 
\begin{align*}
\bigwedge^3 T_X^{\vee} \cong \oO_X, \quad 
H^1(X, \oO_X)=0. 
\end{align*}
The notion of DT invariants is introduced in~\cite{Thom} 
in order to give a holomorphic analogue of 
Casson invariants on real 3-manifolds. 
They are integer valued counting invariants of
stable sheaves on a Calabi-Yau 3-fold $X$, 
and their rank one theory is conjectured to be
equivalent to Gromov-Witten theory~\cite{MNOP}. 
In recent years, the \textit{generalized DT invariants} 
counting not only stable sheaves but also semistable 
sheaves are introduced by Joyce-Song~\cite{JS}, 
Kontsevich-Soibelman~\cite{K-S}. 
They are $\mathbb{Q}$-valued, 
and their definition involves sophisticated techniques on 
motivic Hall algebras. 
It is very difficult to compute or study the generalized 
DT invariants from their definition. 

In this paper, we are interested in generalized
DT invariants counting one dimensional semistable sheaves
on $X$.   
Given data,  
\begin{align*}
n \in \mathbb{Z}, \quad \beta \in H_2(X, \mathbb{Z}),
\end{align*}
the generalized DT invariant is denoted by  
\begin{align}\label{intro:N}
N_{n, \beta} \in \mathbb{Q}. 
\end{align}
The invariant (\ref{intro:N}) counts 
one dimensional semistable sheaves $F$ on $X$ satisfying 
\begin{align*}
\chi(F)=n, \quad [F]=\beta. 
\end{align*}
If $\beta$ and $n$ are coprime, (e.g. $n=1$,)
then the invariant $N_{n, \beta}$ is an integer, and can be defined 
as a holomorphic Casson invariant as in~\cite{Thom}. 
However if $\beta$ and $n$ are not 
coprime, then the invariant (\ref{intro:N}) may not 
be an integer, and its definition requires 
techniques on Hall algebras. 
Roughly speaking, it is defined by the integration of the Behrend function~\cite{Beh}
on the `logarithm' of the moduli stack of semistable sheaves
in the Hall algebra. For the detail, see Definition~\ref{defi:Nn}. 

A particularly important case is when 
$n=1$. In this case, the invariant $N_{1, \beta}$
is nothing but Katz's definition of \textit{genus zero 
Gopakumar-Vafa invariant}~\cite{Katz}. 
In~\cite{JS}, \cite{Tsurvey}, the following 
conjecture is proposed: 
\begin{conj}
{\bf\cite[Conjecture~6.20]{JS}, 
\cite[Conjecture~6.3]{Tsurvey}}\label{conj:mult}
We have the following formula, 
\begin{align}\label{form:mult}
N_{n, \beta}=\sum_{k\ge 1, k|(n, \beta)}
\frac{1}{k^2}N_{1, \beta/k}. 
\end{align}
\end{conj}
The formula (\ref{form:mult}) 
is called a \textit{multiple cover formula} of $N_{n, \beta}$. 
The motivation of the above conjecture is that 
the formula (\ref{form:mult}) is equivalent to 
Pandharipande-Thomas's conjecture~\cite[Conjecture~3.14]{PT} on 
the strong rationality of the generating series of
stable pair invariants. 
(See Subsection~\ref{subsec:conjectural} for some more discussions.)
So far the above conjecture is checked in 
few cases, e.g. $\beta$ is a multiple of 
a homology class of a $(-1, -1)$-curve. 
(cf.~\cite[Example~4.14]{JS}.)
In general the above conjecture seems to be 
difficult to solve, especially due to 
the technical difficulties of the definition of (\ref{intro:N}). 

In the present and the subsequent paper~\cite{Todmu}, 
we approach Conjecture~\ref{conj:mult} by the following ideas: 
\begin{itemize}
\item Introduce the notion of \textit{parabolic stable pairs}
and their counting invariants. They are integer valued and
we can relate them to the invariant $N_{n, \beta}$.  
Consequently the conjectural multiple cover formula (\ref{form:mult})
can be reduced to a certain 
formula relating parabolic stable pair invariants and 
the invariant $N_{1, \beta}$. 
\item There is also a local version of parabolic stable pair theory, and 
the local theory admits an action of the Jacobian group of 
the underlying curve. 
We study the formula (\ref{form:mult}) by the localization 
with respect to the 
actions of the Jacobian group to parabolic stable pair invariants. 
\end{itemize}
The second idea on the Jacobian localizations will be 
pursued in~\cite{Todmu}. 
The present paper is devoted to give a foundation on 
parabolic stable pairs.

\subsection{Parabolic stable pairs}
Recall that a parabolic structure on a 
vector bundle on a curve is given by 
a data of a filtration on some fiber 
together with a parabolic weight~\cite{MeSe}. 
We apply this idea to a one dimensional sheaf $F$ on  
a Calabi-Yau 3-fold $X$. In this case, 
we interpret the `fiber' as 
a tensor product 
of the sheaf $F$ with 
the structure sheaf of a fixed divisor $H$ inside $X$. 

Let $\oO_X(1)$ 
be an ample line bundle on $X$, and 
we set $\omega =c_1(\oO_X(1))$. 
We would like to take a divisor in $X$, 
\begin{align*}
H \in \lvert \oO_X(h) \rvert, \quad 
h \gg 0, 
\end{align*}
so that $H$ intersects with one dimensional 
sheaves $F$ we are interested in transversally. 
In fact if we fix $d \in \mathbb{Z}_{>0}$, then 
we can find $H$ so that 
any one cycle $C$ on 
$X$ with $\omega \cdot [C] \le d$
satisfies, (cf.~Lemma~\ref{lem:trans},)
\begin{align*}\dim H \cap C =0.
\end{align*}
Below we fix such $d$ and $H$. 
We introduce the notion of parabolic 
stable pairs on $X$ to be pairs,
\begin{align}\label{intro:para}
(F, s), \quad s \in F \otimes \oO_{H}, 
\end{align}
where $F$ is a pure one dimensional sheaf on $X$
with $\omega \cdot [F] \le d$. 
The above pair (\ref{intro:para}) should satisfy 
the following stability condition:
\begin{itemize}
\item The sheaf $F$ is an $\omega$-semistable sheaf. 
We denote
\begin{align*}
\mu_{\omega}(F) \cneq \frac{\chi(F)}{\omega \cdot [F]}.
\end{align*}
\item For any surjection $\pi \colon F \twoheadrightarrow F'$
with $\mu_{\omega}(F')=\mu_{\omega}(F)$, we have 
\begin{align*}
(\pi \otimes \oO_{H})(s) \neq 0.
\end{align*}
\end{itemize}

The reason we call a pair (\ref{intro:para}) 
as a parabolic stable pair is that it 
resembles both of some particular 
parabolic vector bundles on 
smooth projective curves~\cite{MeSe}, and
stable pairs studied by Pandharipande-Thomas~\cite{PT}, Joyce-Song~\cite{JS}, 
based on the earlier work of coherent systems by Le Potier~\cite{LeP}.    
See Figure~\ref{fig:one}
for a geometric picture, and 
Subsection~\ref{subsec:relation} for the discussion.

\begin{figure}[htbp]
 \begin{center}
  \includegraphics[width=70mm]{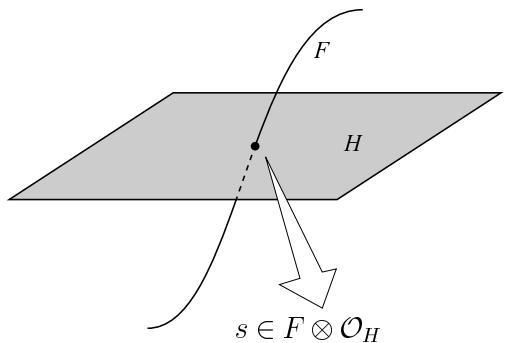}
 \end{center}
 \caption{Parabolic stable pair}
 \label{fig:one}
\end{figure}

Let \begin{align*}
\mM^{\rm{par}}_n(X, \beta) \colon 
\mathrm{Sch}/\mathbb{C} \to \mathrm{Set},
\end{align*}
be the moduli functor which 
assigns an $\mathbb{C}$-scheme $T$ 
to the set of flat families of parabolic
stable pairs $(F, s)$ over $T$ satisfying 
$[F]=\beta$ and $\chi(F)=n$. 
Our first result is the following:
\begin{thm}{\bf [Theorem~\ref{thm:moduli}]}
\label{intro:thmA}
The moduli functor $\mM^{\rm{par}}_n(X, \beta)$
is represented by a 
projective scheme of  
finite type over $\mathbb{C}$, 
denoted by $M_n^{\rm{par}}(X, \beta)$.
\end{thm}

\subsection{Counting invariants of parabolic stable pairs}
By Theorem~\ref{intro:thmA},
we are able to define the integer valued
invariants counting
parabolic stable pairs,
\begin{align*}
\mathrm{DT}^{\rm{par}}_{n, \beta}
\cneq \int_{M_{n}^{\rm{par}}(X, \beta)}
\nu_M d\chi,
\end{align*}
where $\nu_M$ is Behrend's constructible function~\cite{Beh},
\begin{align*}
\nu_M \colon M_n^{\rm{par}}(X, \beta) \to \mathbb{Z}. 
\end{align*} 
For $\mu \in \mathbb{Q}$, the generating series 
of the invariants $\mathrm{DT}^{\mathrm{par}}_{n, \beta}$
is 
defined by 
\begin{align*}
\mathrm{DT}^{\rm{par}}(\mu, d) \cneq 
1+
\sum_{\begin{subarray}{c} 0<\beta \cdot \omega \le d \\
n/\omega \cdot \beta =\mu
\end{subarray}}
\DT_{n, \beta}^{\rm{par}}q^n t^{\beta}. 
\end{align*}
By applying the wall-crossing formula~\cite{JS}, \cite{K-S},
we can express the generating 
series $\mathrm{DT}^{\rm{par}}(\mu, d)$
in terms of $N_{n, \beta}$, where 
$N_{n, \beta}$ is the generalized DT invariant 
(\ref{intro:N}).  The equality of the 
generating series is described in the ring 
$\Lambda_{\le d}$, which is 
a quotient of $\Lambda$, 
\begin{align*}
\Lambda \cneq \bigoplus_{n\in \mathbb{Z}, \beta>0}
\mathbb{Q} q^n t^{\beta},
\end{align*}
by the ideal generated by 
$q^n t^{\beta}$ with $\beta \cdot \omega >d$. 
Here $\beta>0$ means that $\beta$ is a
homology class of an effective one cycle on $X$. 
\begin{thm}{\bf [Theorem~\ref{thm:DTpar}]}\label{thm:DTpar=N}
We have the following formula in $\Lambda_{\le d}$, 
\begin{align}\label{form:DTpar=N}
\mathrm{DT}^{\rm{par}}(\mu, d)
=\prod_{\begin{subarray}{c}
\beta>0, \\
 n/\omega \cdot \beta=\mu
\end{subarray}}
\exp\left( (-1)^{\beta \cdot H -1} 
N_{n, \beta}q^n t^{\beta}  \right)^{\beta \cdot H}. 
\end{align}
\end{thm}
As a corollary, 
Conjecture~\ref{conj:mult}
can be translated into a property on 
the generating series of the invariants of 
parabolic stable pairs. 
\begin{cor}{\bf [Proposition~\ref{prop:mult:para}]}
The formula (\ref{form:mult}) holds for any $(n, \beta)$
with $\beta \cdot \omega \le d$
and $n/\beta \cdot \omega =\mu$ if and only if 
the following formula holds in $\Lambda_{\le d}$, 
\begin{align}\label{cor:form}
\mathrm{DT}^{\rm{par}}(\mu, d)
=\prod_{\begin{subarray}{c}
\beta>0, \\
 n/\omega \cdot \beta=\mu
\end{subarray}}
\left(1-(-1)^{\beta \cdot H} q^n t^{\beta}\right)^{(\beta \cdot H)N_{1, \beta}}
\end{align}
\end{cor}
Note that the equality (\ref{cor:form})
is a relationship between $\mathbb{Z}$-valued invariants. 
There is also a local version of local parabolic 
stable pair theory and the result corresponding to 
Theorem~\ref{thm:DTpar=N}. The detail will be 
discussed in Section~\ref{sec:mult}.

\subsection{Relation to existing works}
The notion of parabolic structures on 
vector bundles on curves
 is introduced by Mehta and Seshadri~\cite{MeSe}. 
It consists of a vector bundle $F$ on a
curve $C$ and a filtration of some fiber of 
$F \to C$, satisfying some stability condition. 
Since its introduction, several generalizations 
have been discussed~\cite{MY}, \cite{BY}.
In~\cite{MY}, Maruyama and Yokogawa
generalize the notion of parabolic 
structures on vector bundles on curves
 to torsion free sheaves on 
arbitrary smooth projective variety, 
and construct their coarse moduli spaces.
The result of Theorem~\ref{intro:thmA}  
is interpreted as a version of 
their result for torsion one dimensional sheaves. 
In~\cite{BY}, Boden and Yokogawa
studies parabolic Higgs bundles over smooth 
projective curves. Since a Higgs bundle is 
interpreted as a torsion sheaf on the total 
space of the canonical line bundle on a curve, 
our work is also interpreted as a 
3-fold version of parabolic Higgs bundles.  
The setting of the 
above works are more general than ours 
in the sense that 
they work under an arbitrary
choice of a filtration and a 
parabolic weight which determine parabolic
structures.  We stick to our 
situation, 
 corresponding to a
 specific choice of a filtration type and 
a parabolic weight,
 as we will not need other choices. 
(See Subsection~\ref{subsec:relation}.)
A complete generalization may be pursued elsewhere.

The idea of parabolic stable pair theory and 
the result of Theorem~\ref{thm:DTpar=N} are very 
similar to Joyce-Song's stable pair theory. 
In~\cite{JS}, Joyce-Song study stable pairs of the form, 
\begin{align}\label{JS:pair}
\oO_X(-n) \to F, 
\end{align}
where $F$ is a coherent sheaf on $X$ and $n\gg 0$, satisfying 
a certain stability condition. Then 
Joyce-Song's stable pair invariants 
are shown to be related to their generalized DT invariants, 
in a way very similar to the formula (\ref{form:DTpar=N}). 
 However there is an advantage of 
parabolic stable pair theory. 
First we note that, 
for each fixed reduced curve $C\subset X$, 
 there is also a 
local version of 
parabolic stable pair invariants, and 
the local version of
the formula (\ref{cor:form}) which 
is enough to show the global formula (\ref{cor:form}).
(See Proposition~\ref{loc:global}.)
The local parabolic stable pair theory admits 
an action of the 
Jacobian group $\Pic^{0}(C)$, 
while the local 
stable pair theory (\ref{JS:pair}) does not. 
Hence we can try to localize by this action. 
It is much 
easier to apply the Jacobian localization 
to the invariants $\DT^{\rm{par}}_{n, \beta}$
rather than $N_{n, \beta}$, as the definition of the 
latter invariants 
involve very complicated techniques on Hall algebras. 
The definition of the invariant $\DT^{\rm{par}}_{n, \beta}$
is much more elementary, and there is no technical 
problem in applying the localization to that invariant. 
This idea works well at least when $C$ has at worst
nodal singularities, and the details will be 
pursued in the subsequent paper~\cite{Todmu}. 
(See Remark~\ref{rmk:final}.)

\subsection{Plan of the paper}
The organization of the paper is as follows. 
In Section~\ref{sec:para},
we introduce parabolic stable pairs and 
construct their moduli spaces. In Section~\ref{sec:wall},
we discuss categorical framework to discuss
parabolic stable pairs, and 
the wall-crossing formula to show 
Theorem~\ref{thm:DTpar=N}. 
In Section~\ref{sec:mult},  
we discuss a relationship
between parabolic stable pair invariants and 
the conjectural multiple cover formula
of generalized DT invariants.

\subsection{Acknowledgement}
The author is grateful to Richard Thomas and 
Jacopo Stoppa for valuable discussions on 
the subject of this paper. 
The author would like to thank the 
Isaac Newton Institute and its program 
`Moduli Spaces', 
during which a part of this work was done. 
This work is supported by World Premier 
International Research Center Initiative
(WPI initiative), MEXT, Japan. This work is also supported by Grant-in Aid
for Scientific Research grant (22684002), 
and partly (S-19104002),
from the Ministry of Education, Culture,
Sports, Science and Technology, Japan.

\section{Moduli spaces of parabolic stable pairs}\label{sec:para}
In this section, we introduce the notion of parabolic 
stable pairs and study their moduli spaces. 
In what follows, $X$ is a smooth projective Calabi-Yau 
3-fold over $\mathbb{C}$, 
\begin{align*}
\bigwedge^{3}T_{X}^{\vee} \cong \oO_X, \quad 
H^1(X, \oO_X)=0. 
\end{align*}
We fix an ample
line bundle $\oO_X(1)$ on $X$
and set $\omega=c_1(\oO_X(1))$. 

\subsection{Semistable sheaves}
First we recall the notion of $\omega$-semistable 
one dimensional sheaves on $X$. 
We set 
\begin{align*}
\Coh_{\le 1}(X) \cneq \{ F \in \Coh(X) : 
\dim \Supp(F) \le 1 \}. 
\end{align*}
For an object $F \in \Coh_{\le 1}(X)$, 
its slope is defined by 
\begin{align*}
\mu_{\omega}(F) \cneq \frac{\chi(F)}{[F] \cdot \omega}. 
\end{align*}
Here $\chi(F)$ is the holomorphic Euler characteristic 
of $F$ and $[F]$ is the fundamental one cycle associated to 
$F$, defined by 
\begin{align}\label{onecycle}
[F] = \sum_{\eta} (\mathrm{length}_{\oO_{X, \eta}}
F) \overline{\{\eta \}}. 
\end{align}
In the above sum, $\eta$ runs all the 
 codimension 
two points in $X$.
 If $[F] \cdot \omega =0$, i.e. $F$ is a zero dimensional 
sheaf, then $\mu_{\omega}(F)$ is defined to be $\infty$. 
\begin{defi}
An object $F \in \Coh_{\le 1}(X)$ is 
$\omega$-(semi)stable if for any 
subsheaf $F' \subset F$, we have the inequality, 
\begin{align*}
\mu_{\omega}(F') <(\le) \mu_{\omega}(F). 
\end{align*}
\end{defi}
Note that any 
one dimensional semistable sheaf $F$ is 
pure, i.e. there is no zero dimensional 
subsheaf in $F$.

\subsection{Definition of parabolic stable pairs}
Let $X$ be a Calabi-Yau 3-fold and $\omega=c_1(\oO_X(1))$ 
as in the previous 
subsection. 
In this subsection, 
we also fix another divisor,
\begin{align}\label{divisor}
H \subset X. 
\end{align}
We introduce the notion 
of parabolic stable pairs to
be pairs of 
one dimensional sheaf $F$ together 
with a data 
\begin{align}\label{data}
s \in F \otimes_{\oO_X} \oO_{H},
\end{align}
satisfying a certain stability condition. 
In what follows, we write $F \otimes_{\oO_X} \oO_H$
as $F\otimes \oO_H$ for simplicity.

In order to 
make the notation (\ref{data}) well-defined, 
 we 
need a transversality condition of the 
support of $F$. 
Namely for a one cycle $C$ on $X$, 
we say $C$ \textit{intersects with $H$ transversally} 
if $H\cap C$ is zero dimensional, or 
equivalently, any irreducible component $C'\subset C$
is not contained in $H$. 
(We allow the multiplicity of $H \cap C$.)
For  
a one dimensional sheaf $F$ on $X$, 
we say that $F$ \textit{intersects with $H$ 
transversally} if 
the one cycle $[F]$
given by (\ref{onecycle}) intersects with 
$H$ transversally. 
Then $F \otimes \oO_{H}$ is a zero dimensional 
sheaf on $X$, and we identify 
$F \otimes \oO_{H}$ with its global 
section $\Gamma(X, F \otimes \oO_{H})$
as a finite dimensional $\mathbb{C}$-vector space. 
 
\begin{defi}\label{defi:para}
For a fixed divisor $H$ on $X$ as above, 
a parabolic stable pair is defined to be a pair 
\begin{align}\label{para:pair}
(F, s), \quad s \in F \otimes \oO_{H},
\end{align}
such that the following conditions are satisfied. 
\begin{itemize}
\item The sheaf $F$ is a one dimensional 
$\omega$-semistable sheaf on $X$. 
\item The one cycle $[F]$ intersects with $H$ transversally. 
\item For any surjection $F \stackrel{\pi}{\twoheadrightarrow} F'$
with $\mu_{\omega}(F)=\mu_{\omega}(F')$, we have 
\begin{align*}
(\pi \otimes \oO_{H})(s) \neq 0. 
\end{align*}
\end{itemize}
\end{defi} 
For two parabolic stable pairs, 
\begin{align*}
(F_i, s_i), \quad i=1, 2, 
\end{align*}
we say they are \textit{isomorphic} if there is 
an isomorphism of sheaves,
\begin{align*}
g \colon F_1 \stackrel{\sim}{\to} F_2,
\end{align*}
satisfying $(g\otimes \oO_H)(s_1) =s_2$. 
In the following local $(-1, -1)$-curve example, 
the isomorphism classes of parabolic 
stable pairs can be easily 
classified. 
\begin{exam}\label{exam:-1}
Let $f \colon X \to Y$ be a 
birational contraction which contracts a 
$(-1, -1)$-curve $C \subset X$, i.e. 
\begin{align*}
C \cong \mathbb{P}^1, \quad 
N_{C/X} \cong \oO_C(-1) \oplus \oO_C(-1). 
\end{align*}
Suppose that a divisor $H \subset X$ intersects with 
$C$ at a one point $p\in C$. 
Let $(F, s)$ be a parabolic stable pair 
on $X$ with $F$ supported on $C$. 
It is easy to see that 
any semistable sheaf on $X$ supported on $C$
is an $\oO_{C}$-module, hence we have 
\begin{align*}
F \cong \oO_{C}(a)^{\oplus r},
\end{align*}
for some $a \in \mathbb{Z}$ and $r\in \mathbb{Z}_{\ge 1}$. 
By applying an action of 
$\Aut(F) \cong \mathbb{\GL}(r, \mathbb{C})$, 
we 
may assume that $s \in F \otimes \oO_{H} \cong \oO_p(a)^{\oplus r}$
 is of the form, 
\begin{align*}
s=(s_1, 0, \cdots, 0) \in \oO_{p}(a) \oplus \cdots \oplus \oO_{p}(a), 
\end{align*}
with $s_1 \neq 0$. 
However if $r \ge 2$, then the surjection 
\begin{align*}
F \cong \oO_C(a)^{\oplus r} \ni (x_1, \cdots, x_r) 
\mapsto x_r \in \oO_C(a), 
\end{align*}
violates the third condition of Definition~\ref{defi:para}. 
Therefore parabolic stable pairs supported on $C$
are given by 
\begin{align*}
(\oO_C(a), s), \ a \in \mathbb{Z}, \
s\in \oO_{p}(a) \setminus \{0\}. 
\end{align*}
\end{exam}
The next example shows that there is a 
parabolic stable pair $(F, s)$
with $F$ strictly $\omega$-semistable. 
\begin{exam}
Let $i\colon C \hookrightarrow X$ be an 
irreducible curve 
whose arithmetic genus is one, i.e. 
\begin{align*}
H^1(C, \oO_C)=\mathbb{C}. 
\end{align*}
Suppose that there is a 
divisor $H \subset X$ which 
intersects with $C$ at a one point 
$p \in C$. 
There is a non-trivial
exact sequence of sheaves on $C$, 
\begin{align*}
0 \to \oO_C \stackrel{j_1}{\to} E \stackrel{j_2}{\to} \oO_C \to 0. 
\end{align*}
By taking the pushforward $i_{\ast}$
and the tensor product with $\oO_{H}$, 
we obtain the exact sequence of vector spaces, 
\begin{align*}
0 \to \oO_p \stackrel{j_{1, p}}{\to}
 i_{\ast}E \otimes \oO_H \stackrel{j_{2, p}}{\to} \oO_p \to 0. 
\end{align*}
Suppose that $s\in i_{\ast}E \otimes \oO_H$ satisfies 
$j_{2, p}(s)\neq 0$. Then it is easy to see that the pair 
\begin{align*}
(i_{\ast}E, s), 
\end{align*}
is a parabolic stable pair. 
Note that $i_{\ast}E$ is $\omega$-semistable but 
not $\omega$-stable. 
\end{exam}

\subsection{Relation to parabolic vector bundles, 
PT stable pairs}\label{subsec:relation}
The notion of parabolic stable pairs resembles 
particular parabolic vector bundles on 
smooth projective curves~\cite{MeSe}.
Recall that for a smooth projective curve $C$
 and a vector bundle 
$E$ on $C$, a quasi-parabolic structure 
on $E$ at $p\in C$ is given by a filtration, 
\begin{align}\label{filt}
0=F_0 \subset F_1 \subset \cdots \subset F_n =E_{p}, 
\end{align}
where $E_p$ is the fiber of $E \to C$ at 
$p$. A parabolic structure is given by choosing 
a parabolic weight, 
\begin{align*}
0<\alpha_{n}<\alpha_{n-1}< \cdots <\alpha_{1}<1, 
\end{align*}
giving a stability condition on 
quasi-parabolic vector bundles. 

Suppose that the curve $C$ is embedded into 
a Calabi-Yau 3-fold $X$
by $i \colon C \hookrightarrow X$, and 
a divisor $H \subset X$ 
scheme theoretically 
intersects with $C$ at a point 
$p \in C$. 
Then we have the isomorphism 
of $\mathbb{C}$-vector spaces, 
\begin{align*}
i_{\ast}E \otimes \oO_H \cong E_p, 
\end{align*}
 and 
 $s \in i_{\ast}E \otimes \oO_{H}$
determines a one dimensional subspace 
in $E_{p}$. 
Thus a parabolic stable pair 
\begin{align*}
(i_{\ast}E, s), \quad s \in i_{\ast}E \otimes \oO_{H},
\end{align*} determines 
a filtration (\ref{filt})
with $F_1$ one dimensional
and $F_2=E_{p}$. 
The stability condition 
in Definition~\ref{defi:para} corresponds to a choice 
of a parabolic weight $0<\alpha_2<\alpha_1 \ll 1$. 

On the other hand, if $F$ is a pure one dimensional 
sheaf whose support intersects with $H$ transversally, 
a pair (\ref{para:pair}) 
is also interpreted to be a pair, 
\begin{align}\label{pair:N}
(F, s), \quad 
N_{H/X}[-1] \stackrel{s}{\to} F. 
\end{align}
Here $N_{H/X}$ is the normal bundle of $H$ 
to $X$,
and $[-1]$ is a $(-1)$-shift in the derived category 
of coherent sheaves on $X$. 
The above observation follows from the following lemma. 
\begin{lem}\label{lem:Ext}
Let $F$ be a pure one dimensional sheaf 
whose support intersects with $H$ 
transversally. Then we have the canonical 
isomorphisms, 
\begin{align*}
\Ext_{X}^{i}(N_{H/X}, F)
\cong \left\{ \begin{array}{cc}
F\otimes \oO_{H}, & i = 1, \\
0, & i\neq 1. 
\end{array}  \right.
\end{align*}
\end{lem}
\begin{proof}
We have the following local to global 
spectral sequence, 
\begin{align}\label{Epq}
E_2^{p, q} \cneq 
H^p(X, \eE xt_{X}^{q}(N_{H/X}, F)) \Rightarrow \Ext_{X}^{i}(N_{H/X}, F). 
\end{align}
By the exact sequence, 
\begin{align*}
0 \to \oO_{X} \to \oO_X(H) \to N_{H/X} \to 0, 
\end{align*}
we
have $\eE xt_X^i(N_{H/X}, F)=0$ for $i\neq 0, 1$ and 
the exact sequence, 
\begin{align*}
0 \to \hH om(N_{H/X}, F) \to F(-H) \stackrel{\tau}{\to}F\to 
\eE xt^1_{X}(N_{H/X}, F) \to 0. 
\end{align*}
By the transversality assumption, the map $\tau$ is an
isomorphism in dimension one. Also since $F$ is pure, 
the morphism $\tau$ is injective, hence we have
\begin{align*}
\hH om(N_{H/X}, F)=0, \quad \eE xt^1_X(N_{H/X}, F)
\cong F\otimes \oO_H.
\end{align*}
Since $F\otimes \oO_{H}$ is a zero dimensional sheaf,  
the spectral sequence (\ref{Epq}) degenerates 
and the assertion holds. 
\end{proof}
\begin{rmk}
Note that 
we have the canonical isomorphism,
\begin{align*}
N_{H/X} \stackrel{\sim}{\to} \oO_H(H),
\end{align*}
 since 
$H$ is a divisor. 
By taking a trivialization of $\oO_H(H)$ 
near the intersection $\Supp(F) \cap H$, 
the pair (\ref{pair:N}) is also interpreted to be a pair, 
\begin{align}\label{pair:N2}
(F, s), \quad \oO_{H}[-1] \stackrel{s}{\to} F.
\end{align}
However the correspondence between (\ref{para:pair})
and (\ref{pair:N2}) is not canonical, so
we interpret (\ref{para:pair}) as a pair (\ref{pair:N}), not (\ref{pair:N2}). 
\end{rmk}

The pair (\ref{pair:N})
also resembles 
stable pairs discussed in~\cite{PT}, \cite{JS}, 
or more generally 
coherent systems~\cite{LeP}. 
For instance, a PT stable pair~\cite{PT}
consists of a pair, 
\begin{align}\label{PT/pair}
\oO_X \stackrel{s}{\to} F, 
\end{align} 
where $F$ is a pure one dimensional sheaf and $s$
is surjective in dimension one. 
Our pair (\ref{pair:N}) replaces $\oO_X$
by $N_{H/X}[-1]$. 
(But the stability conditions are different.)

In subsection~\ref{subsec:category}, 
we will interpret the stability condition in 
Definition~\ref{defi:para} in terms of 
a stability condition in the category of pairs
of the form
\begin{align}\label{pair:r}
N_{H/X}^{\oplus r}[-1] \to F, 
\end{align}
and proceed the arguments as given in~\cite[Section~13]{JS}. 
In fact using the description of 
parabolic stable pairs in terms of stable 
objects in the category of pairs (\ref{pair:r}),
we can show the following. 
\begin{lem}\label{lem:aut}
For a parabolic stable pair $(F, s)$, 
let $\Aut(F, s)$ be 
\begin{align*}
\Aut(F, s)\cneq \{ g \in \Aut(F)  :
(g\otimes \oO_H)(s)=s\}. 
\end{align*}
Then we have $\Aut(F, s)=\{ \emph{\id}_{F}\}$.
\end{lem} 
\begin{proof}
The proof will be given in 
Corollary~\ref{cor:later}. 
\end{proof}

\subsection{Families of parabolic stable pairs}
\label{subsec:Family}
The moduli space of parabolic 
stable pairs is defined to be the scheme 
representing the functor, 
\begin{align}\label{funct}
\mM_n^{\rm{par}}(X, \beta) \colon 
\mathrm{Sch}/\mathbb{C} \to \mathrm{Set}, 
\end{align}
which assigns an $\mathbb{C}$-scheme 
$T$ to the isomorphism classes
of families of parabolic stable 
pairs $(F, s)$ satisfying 
a numerical condition, 
\begin{align}\label{num:cond}
[F]=\beta, \quad \chi(F)=n. 
\end{align}
Here $\beta \in H_2(X, \mathbb{Z})$
and, 
by an abuse of notation, 
 $[F]$ is the homology class of 
the fundamental one cycle (\ref{onecycle}). 

More precisely, the functor (\ref{funct}) is defined 
as follows. 
The set of $T$-valued points of $\mM_n^{\rm{par}}(X, \beta)$
consists of isomorphism classes of pairs, 
\begin{align*}
(\fF, s), 
\end{align*}
satisfying the following conditions. 
\begin{itemize}
\item $\fF \in \Coh(X\times T)$ is a flat family 
of one dimensional coherent sheaves on $X$ satisfying 
(\ref{num:cond}), the first and the second conditions 
in Definition~\ref{defi:para}. 
\item $s$ is a global section, 
\begin{align*}
s \in \Gamma(T, \pi_{T\ast}(\fF \otimes \pi_{X}^{\ast}\oO_{H})),
\end{align*}
such that for each closed point $t\in T$, 
the pair $(\fF_t, s_t)$ is a parabolic 
stable pair on $X$. 
Here $\pi_{X}$ and $\pi_{T}$ are projections
from $X\times T$ onto the corresponding factors. 
\end{itemize}
Here we explain the above second condition. 
Note that the support of 
$\fF \otimes \pi_{X}^{\ast}\oO_{H}$ is 
finite over $T$ by the first condition, 
hence we have 
\begin{align*}
R^{i}\pi_{T\ast}(\fF \otimes \pi_{X}^{\ast}\oO_{H}) =0, 
\end{align*}
for $i>0$. By the base change theorem, 
$\pi_{T\ast}(\fF \otimes \pi_{X}^{\ast}\oO_H)$
is a vector bundle on $T$, 
satisfying 
\begin{align}\label{vect:bund}
\pi_{T\ast}(\fF \otimes \pi_{X}^{\ast}\oO_H) \otimes k(t) 
\cong \fF_{t} \otimes \oO_{H}, 
\end{align}
for any closed point $t\in T$. The above 
second condition is now clear from (\ref{vect:bund}). 

\begin{rmk}
The moduli functor $\mM_n^{\rm{par}}(X, \beta)$
also depends on $\omega$ and $H$. 
We omit these in $\mM_n^{\rm{par}}(X, \beta)$ in order
to simplify the notation. 
\end{rmk}

In general, the
 moduli space of parabolic stable pairs may not be 
compact, since the semistable sheaf $F$ may 
degenerate to a sheaf which lies in $H$. 
In order to avoid such a case, we need to choose 
a suitable divisor 
\begin{align*}
H \in \lvert \oO_X(h) \rvert, \quad h \gg 0, 
\end{align*}
which intersects with one dimensional 
semistable sheaves we are interested in transversally. 
(Recall that $\oO_X(1)$ is an ample line bundle 
on $X$ with $c_1(\oO_X(1))=\omega$.) 
Such a divisor $H$ can be found, once we 
fix a positive integer $d\in \mathbb{Z}_{\ge 1}$
and consider only one dimensional semistable sheaves $F$
satisfying $\omega \cdot [F] \le d$. 
In fact we have the following lemma. 
\begin{lem}\label{lem:trans}
For each $d \in \mathbb{Z}_{>0}$, there is a 
divisor $H \in \lvert \oO_X(h) \rvert$
for $h \gg 0$, (depending on $d$,) 
which intersects with any 
one cycle $C$ on $X$
satisfying $\omega \cdot C \le d$,
transversally. 
\end{lem}
\begin{proof}
Let $\mathrm{Chow}_{\le d}(X)$ be the Chow variety 
parameterizing one cycles $C$ on $X$ with 
$\omega \cdot C \le d$. Let $\oO_X(1)$ be a 
very ample line bundle on $X$. 
We would like to find $h\gg 0$ and 
an element $H \in \lvert \oO_X(h) \rvert$
which satisfies the desired condition. 
We define the set $Z$ to be the subset, 
\begin{align*}
Z \subset \Chow_{\le d}(X) \times \lvert \oO_X(h) \rvert,
\end{align*}
consisting of $([C], H)$ 
such that there is an irreducible component 
$C' \subset C$ with $C' \subset H$. 
A desired $H$ can be found if 
the projection 
\begin{align}\label{dominant}
Z \to \lvert \oO_X(h) \rvert,
\end{align}
is not a dominant map. 

By the boundedness of the 
Chow variety, we may assume that 
\begin{align}\label{mayassume}
H^1(X, I_{C'}(h)) =
H^1(C', \oO_C'(h))=0, 
\end{align}
for any 
one cycle $[C] \in \Chow_{\le d}(X)$ and 
an irreducible component 
$C' \subset C$. 
Also the Euler characteristic 
$\chi(\oO_{C'})$ for the above $C'$
is bounded below, say $\chi(\oO_{C'}) \ge D$. 
Let $e$ be the smallest number of 
$\deg \oO_{C}(1)$ among curves 
$C$ on $X$, and $Z_{[C]}$
the fiber of 
the projection 
\begin{align*}
Z \to \Chow_{\le d}(X),
\end{align*} at 
a one cycle $[C]$. 
If $C' \subset C$ is an irreducible component, 
we have the exact sequence, 
\begin{align}\label{ex:seq}
0 \to I_{C'} \otimes \oO_X(h) \to 
\oO_X(h) \to \oO_{C'}(h) \to 0. 
\end{align}
Using (\ref{mayassume}),
(\ref{ex:seq}) and the Riemann-Roch theorem,
the dimension of $Z_{[C]}$ is evaluated as 
\begin{align*}
\dim Z_{[C]} &=\dim \lvert \oO_X(h) \rvert 
- \dim H^0(C', \oO_{C'}(h)) \\
&\le \dim \lvert \oO_X(h) \rvert 
-he -D. 
\end{align*}
Therefore
we have 
\begin{align}\label{ineq:chow}
\dim Z \le \dim \Chow_{\le d}(X) + \dim \lvert \oO_X(h) \rvert
-he -D. 
\end{align}
Suppose that the map (\ref{dominant})
is a dominant map. 
Then we have 
$\dim Z \ge \dim \lvert \oO_X(h) \rvert$, hence 
the inequality (\ref{ineq:chow}) implies 
\begin{align}\label{ineq:cont}
he \le \dim \Chow_{\le d}(X) -D. 
\end{align}
The above inequality is not satisfied 
for $h\gg 0$. If $h>0$ 
does not satisfy the above inequality, 
then we can find 
a desired $H \in \lvert \oO_X(h) \rvert$. 
\end{proof}
The next subsections
are devoted to show the following theorem. 
\begin{thm}\label{thm:moduli}
For $d \in \mathbb{Z}_{>0}$, 
choose a divisor $H \in \lvert \oO_X(h) \rvert$ satisfying 
the condition of Lemma~\ref{lem:trans}. Then
for $n\in \mathbb{Z}$ and $\beta \in H_2(X, \mathbb{Z})$
with $\omega \cdot \beta \le d$, 
 the 
functor $\mM_n^{\rm{par}}(X, \beta)$
is represented by a projective $\mathbb{C}$-scheme
$M_n^{\rm{par}}(X, \beta)$.
\end{thm}

\subsection{Construction of the moduli space}\label{subsec:const}
In this subsection, we give a construction of the 
moduli space of parabolic stable pairs. 
As we discussed in Subsection~\ref{subsec:relation}, 
the notion of parabolic stable pairs resemble 
both of parabolic vector bundles on 
curves and PT stable pairs (coherent systems). 
The moduli space of parabolic 
vector bundles is constructed in~\cite{MeSe}, 
and its generalization to parabolic sheaves 
for torsion free sheaves on arbitrary algebraic 
varieties is discussed in~\cite{MY}. 
On the other hand, the moduli space of coherent 
systems is constructed in~\cite{LeP}, and 
its construction is
simplified for PT stable pairs in~\cite{StTh}. 
Our strategy is to imitate
the construction of the moduli space of PT stable pairs 
by applying the arguments in~\cite[Section~3]{StTh}. 

Let us take $d\in \mathbb{Z}$, 
$H \in \lvert \oO_X(h) \rvert$ 
and $n, \beta$ as in the statement of 
Theorem~\ref{thm:moduli}. 
We first note that, by the boundedness of 
semistable sheaves, 
there is $m>0$ 
such that $m'\ge m$ implies
$H^i(X, F(m'))=0$ for $i>0$ 
and $F(m')$ is globally generated
for any $\omega$-semistable one dimensional 
sheaf 
$F$ with
$[F]=\beta$ and $\chi(F)=n$. 
However for later use, we take $m>0$ satisfying the following 
stronger condition.
Note that we have 
\begin{align*}
\mu_{\omega}(F(1))=1+\mu_{\omega}(F).
\end{align*}
Also the set of $\omega$-semistable one dimensional 
sheaves $F'$ satisfying 
\begin{align}\label{F'}
\omega \cdot [F'] \le d,  \quad
\frac{n}{\omega \cdot \beta} \le \mu_{\omega}(F')
<1+\frac{n}{\omega \cdot \beta},
\end{align}
is bounded. Hence 
we can take $m >0$ such that 
\begin{align}\label{vanish:F'}
H^i(X, F'(m))=0, \quad i>0,
\end{align}
and $F'(m)$ is globally generated for any
$\omega$-semistable sheaves $F'$ satisfying 
$\omega \cdot [F'] \le d$ and $\mu_{\omega}(F') \ge n/\omega \cdot \beta$.  
Below we fix such $m>0$. 
 
Let $F$ be an $\omega$-semistable 
one dimensional sheaf with $[F]=\beta$
and $\chi(F)=n$. 
By the above choice of $m$, we have 
\begin{align}\label{Hilb:poly}
\chi_{F}(m) &\cneq \dim H^0(X, F(m)) \\
\notag
&= m(\omega \cdot \beta) +n, 
\end{align}
by the Riemann-Roch theorem. 
Below we write 
\begin{align*}
\chi_{n, \beta}(m) \cneq m(\omega \cdot \beta)+n,
\end{align*}
and set $V$ to be the $\mathbb{C}$-vector 
space of dimension
$\chi_{n, \beta}(m)$. Then 
by the vanishing (\ref{vanish:F'}), 
for such an $\omega$-semistable sheaf $F$, 
we have a surjection, 
\begin{align}\label{surjection}
V \otimes \oO_X(-m) \twoheadrightarrow F, 
\end{align}
such that the induced morphism 
\begin{align}\label{induced}
V \to H^0(X, F(m))
\end{align}
 is an isomorphism. 

We consider the Quot-scheme, 
\begin{align}\label{Quot}
Q \cneq \mathrm{Quot}(V\otimes \oO_X(-m), n, \beta), 
\end{align}
which parameterizes quotients 
$V\otimes \oO_X(-m) \twoheadrightarrow F$
with $F$ one dimensional sheaf 
satisfying $[F]=\beta$ and $\chi(F)=n$. 
The scheme $Q$ contains the open subset, 
\begin{align}\label{UQ}
U \subset Q, 
\end{align}
which corresponds to quotients (\ref{surjection})
such that $F$ is $\omega$-semistable and the 
induced morphism (\ref{induced}) is an isomorphism. 
Note that outside $U$
the 
sheaf $F$ may no longer be semistable nor a 
pure sheaf. 

Let 
\begin{align*}
\left( V\otimes \oO_X(-m) \right) 
\boxtimes \oO_{Q} \twoheadrightarrow \mathbb{F},
\end{align*}
be the universal quotient on $X\times Q$ and 
we denote by $\mathbb{F}_{U}$ the restriction of 
$\mathbb{F}$ to $X\times U$. 
By the arguments in the previous subsection
and Lemma~\ref{lem:trans}, 
we have the vector bundle $R_U$ on $U$, 
\begin{align}\label{RU}
R_U \cneq \pi_{U\ast}(\mathbb{F}_{U} \otimes \oO_{H\times U})
\to U. 
\end{align} 
Note that pairs $(F, s)$ satisfying the 
first and the second conditions in Definition~\ref{defi:para}
and $[F]=\beta$, $\chi(F)=n$
bijectively correspond to closed points in $R_U$
up to the action of $\mathrm{GL}(V)$ on $R_U$.
Let 
\begin{align*}
R_{U}^{\rm{s}} \subset R_{U},
\end{align*}
be the subset corresponding to parabolic 
stable pairs. 
Suppose for instance that $R_U^{\rm{s}}$ is an open subset of 
$R_U$. Then 
the resulting moduli space can be
constructed to be the quotient space,
\begin{align}\label{quot:space}
R_U^{\rm{s}}/\mathrm{GL}(V) =
R_U^{'\rm{s}}/\mathrm{SL}(V), 
\end{align}
where $R_U^{'\rm{s}}$ is the quotient space of 
$R_U^{\rm{s}}$ by the diagonal subgroup 
$\mathbb{C}^{\ast} \subset \mathrm{GL}(V)$. 
Note that we have 
\begin{align}\label{PRU}
R_U^{'\rm{s}} \subset \mathbb{P}(R_U), 
\end{align} 
where $\mathbb{P}(R_U)$ is the projectivization of 
$R_U$. 

By Lemma~\ref{lem:aut}, the action of 
$\mathrm{SL}(V)$ on $R_U^{'\rm{s}}$ is free, 
hence the quotient space (\ref{quot:space})
is at least an algebraic space 
of finite type,
 once we prove the openness of 
parabolic stable pair locus. 
In fact we show that $R_U^{'\rm{s}}$
coincides with a GIT stable locus 
in a certain projective compactification of 
$\mathbb{P}(R_U)$, 
hence in particular $R_U^{'\rm{s}}$
is an open subset of $\mathbb{P}(R_U)$. 

\subsection{Compactification of $\mathbb{P}(R_U)$}
In order to interpret $R_U^{'\rm{s}}$
as a GIT stable locus, we embedded $\mathbb{P}(R_U)$
into a product of a Grassmannian and the Quot scheme $Q$.  
Let $V\otimes \oO_X(-m) \twoheadrightarrow F$ be a quotient 
corresponding to a point in $Q$, 
and $V' \subset V$ a linear subspace. 
Let $F' \subset F$ be the subsheaf generated by $V'$. 
We take an exact sequence, 
\begin{align*}
0 \to G' \to V'\otimes \oO_X(-m) \to F' \to 0. 
\end{align*}
By the boundedness of Quot scheme and the subspaces 
$V' \subset V$, we can take $m'> 2m$ satisfying 
\begin{align}\label{vanish}
H^1(H, G'\otimes \oO_{H}(m'))=0,
\end{align}
and $\oO_X(m')$ is globally generated. 
Below such $m'>2m$ is also fixed. 

Suppose that $V\otimes \oO_X(-m) \twoheadrightarrow F$
corresponds to a point in $U$. 
By the vanishing (\ref{vanish}) for $V'=V$, we have the surjection, 
\begin{align}\label{VK}
V \otimes K \twoheadrightarrow F\otimes \oO_{H}(m'), 
\end{align}
where $K=H^0(H, \oO_H(m'-m))$. 
Also since $\oO_X(m')$ is globally generated, 
the natural morphism, 
\begin{align}\label{nat:mor}
F \to F(m') \otimes H^0(X, \oO_X(m'))^{\vee},
\end{align}
is injective whose cokernel is a pure one dimensional sheaf. 
By setting $M=H^0(X, \oO_X(m'))^{\vee}$, 
tensoring $M, \oO_H$ with (\ref{VK}), (\ref{nat:mor})
respectively, 
and composing them, 
we obtain the surjections, 
\begin{align*}
V \otimes K \otimes M &\twoheadrightarrow F \otimes \oO_H(m')
 \otimes M \\
&\twoheadrightarrow 
\left(F \otimes \oO_H(m') \otimes M \right)/
\left(F \otimes \oO_H \right).
\end{align*}
Hence a one dimensional subspace in $F \otimes \oO_H$
corresponds to a $d_M(\beta \cdot H)-1$-dimensional 
quotient of $V \otimes K \otimes M$, 
where $d_M=\dim M$.

The above argument yields the embedding, 
\begin{align*}
\mathbb{P}(R_U) \hookrightarrow G(V\otimes K \otimes M, d_M(\beta \cdot H) -1) 
\times Q.
\end{align*}
Here for a vector space $V$, we denote by 
$G(V, m)$ the Grassmannian parameterizing quotients 
$V\twoheadrightarrow V'$ with $\dim V'=m$. 
We define $R$ to be the Zariski closure, 
\begin{align*}
R \cneq \overline{\mathbb{P}(R_U)} \subset
G(V\otimes K \otimes M, d_M(\beta \cdot H)-1) \times Q.  
\end{align*}
The next purpose is to give a polarization of $R$. 
Let $V\otimes \oO_X(-m) \twoheadrightarrow F$ be a
quotient giving a closed point in $Q$. 
By taking $l\gg 0$, we have the surjection, 
\begin{align*}
V \otimes W \twoheadrightarrow H^0(X, F(l)), 
\end{align*}
where $W=H^0(X, \oO_X(l-m))$. 
The above surjections induce the embedding, 
\begin{align*}
Q \hookrightarrow G(V\otimes W, \chi_{n, \beta}(l)).
\end{align*}
Recall that any Grassmannian 
$G(V, m)$ is embedded into
$\mathbb{P}(\wedge^m V)$
 via Pl$\ddot{\rm{u}}$cker embedding. 
Pulling back $\oO(1)$ on $\mathbb{P}(\wedge^m V)$, 
we have a $\mathrm{GL}(V)$-equivariant 
polarization $\oO_{G(V, m)}(1)$ on $G(V, m)$. 

For simplicity we write
\begin{align*}
G_1=G(V\otimes K\otimes M, d_M(\beta \cdot H) -1), \
G_2=G(V\otimes W, \chi_{n, \beta}(l)).
\end{align*}
By the above argument, we have 
$R \subset G_1 \times G_2$, hence 
there is a following $\SL(V)$-equivariant ample line bundle
$\lL$ on $R$, 
\begin{align*}
\lL \cneq \oO_{G_1}(1) \boxtimes \oO_{G_2}(1)|_{R}. 
\end{align*}
In the next subsection, 
for $l\gg 0$, 
we see that the
GIT stable locus in $R$ w.r.t. $\lL$
 coincides with 
$R^{'\rm{s}}_{U}$. 

\subsection{Hilbert-Mumford criterion}
We apply the Hilbert-Mumford criterion~\cite{Mum} 
to investigate the 
GIT stable locus with respect to $(R, \lL)$. 
Let $\lambda$ be a one parameter subgroup, 
\begin{align*}
\lambda \colon \mathbb{C}^{\ast} \to \mathrm{SL}(V). 
\end{align*}
Then $\lambda$ corresponds to the grading of $V$, 
\begin{align*}
V =\bigoplus_{k\in \mathbb{Z}}V_k,
\end{align*}
where $V_k$ is the space of weight $k$. 
Let us take a closed point in $R$, 
\begin{align*}
(V\otimes K \otimes M \stackrel{\phi}{\twoheadrightarrow} A, 
V\otimes \oO_X(-m) \stackrel{\rho}\twoheadrightarrow F) 
\in R \subset G_1 \times Q. 
\end{align*}
For simplicity, we write the above point by 
$(\phi, \rho)$. 
We set $V_{\le k} \cneq \oplus_{j\le k}V_j$ and
$A_{\le k}$, $F_{\le k}$ to be
\begin{align*}
A_{\le k} &\cneq \phi(V_{\le k} \otimes K\otimes M), \\
F_{\le k} &\cneq \rho(V_{\le k} \otimes \oO_X(-m)).
\end{align*}
Let $A_{k} \cneq A_{\le k}/A_{\le k-1}$
and $F_{k} \cneq F_{\le k}/F_{\le k-1}$. 
By the construction, we have the surjections, 
\begin{align*}
\phi_k &\colon V_k \otimes K \otimes M \twoheadrightarrow A_k, \\
\rho_k & \colon V_k \otimes \oO_X(-m) \twoheadrightarrow F_k. 
\end{align*}
The following proposition follows from an
argument of the application of 
Hilbert-Mumford criterion to the moduli of sheaves. 
For instance, see~\cite[Lemma~4.4.3, Lemma~4.4.4]{Hu}, \cite[Lemma~3.12]{StTh}.
\begin{prop}
In the above situation, we have 
\begin{align*}
\lim_{t\to 0}\lambda(t)
\cdot (\phi, \rho)=
\left(\oplus_{k}\phi_k, \oplus_{k} \rho_k  \right). 
\end{align*}
The Hilbert-Mumford weight $\mu^{\lL}((\phi, \rho), \lambda)$
is given by 
\begin{align*}
\frac{1}{\dim V}\sum_{k}\left\{
 \dim V \left(\chi_{F\le k}(l)+\dim A_{\le k} \right)
-\dim V_{\le k}\left(\chi_{F}(l)+\dim A \right)\right\}. 
\end{align*}
Here $\chi_{F}(l)$ is the Hilbert polynomial (\ref{Hilb:poly}). 
\end{prop}
By the Hilbert-Mumford criterion, a point 
$(\phi, \rho) \in R$ is GIT (semi)stable if 
for any non-trivial one parameter subgroup $\lambda$
as above, we have $\mu^{\lL}((\phi, \rho), \lambda) >(\ge) 0$. 
This condition is equivalent to that, 
for any proper subspace $V' \subset V$, we have 
\begin{align}\label{V'V}
\dim V \left(\chi_{F'}(l)+\dim A' \right)
-\dim V'\left(\chi_{F}(l)+\dim A \right) >(\ge) 0. 
\end{align}
Here $A' \subset A$ and $F' \subset F$ are 
subspace and the subsheaf generated by 
$V'\otimes K \otimes M$ and $V'\otimes \oO_X(-m)$ respectively. 
In fact if we have such a subspace $V' \subset V$, then 
there is a one parameter subgroup $\lambda$ 
whose induced grading on $V$ is
$V_{\le -1}=0$, $V_{\le 0}=V'$
and $V_{\le 1}=V$.

Let $R^{\lL\mbox{-}\rm{s}}$ and $R^{\lL\mbox{-}\rm{ss}}$ be
the GIT stable and semistable locus in $R$, 
and $R_{U}^{'\rm{s}}$ the subspace of $\mathbb{P}(R_U)$
given in (\ref{PRU}). 
We prove the following proposition. 
\begin{prop}\label{prop:GIT}
For $(\phi, \rho) \in R$, 
the following three conditions are equivalent. 

(i) We have $(\phi, \rho) \in R_{U}^{'\rm{s}}$. 

(ii) We have $(\phi, \rho) \in R^{\lL\mbox{-}\rm{s}}$. 

(iii) We have $(\phi, \rho) \in R^{\lL\mbox{-}\rm{ss}}$. 
\end{prop}
\begin{proof}
(i) $\Rightarrow$ (ii) : 
Suppose that $(\phi, \rho) \in R_{U}^{'\rm{s}}$, i.e. 
it determines a parabolic stable pair. 
We show that the LHS of (\ref{V'V}) is positive for 
any proper subspace $V' \subset V$. 
For simplicity, we write 
$\chi_F(l)=rl+n$ and $\chi_{F'}(l)=r'l+n'$. 
Note that $r=\beta \cdot \omega$. 
Since we are taking $l\gg 0$, the assertion holds if 
either
\begin{align}\label{either1}
r' \dim V  > r \dim V', 
\end{align}
or $r' \dim V =r \dim V'$ and 
\begin{align}\label{either2}
\dim V(n'+\dim A')> \dim
 V'(n+\dim A), 
\end{align}
holds. 
Also since $V'$ is a subspace of 
$H^0(X, F'(m))$, we may assume that 
$V' =H^0(X, F'(m))$. 
Then we have 
\begin{align*}
\frac{r \dim V'}{r'\dim V}
= \frac{m+\mu_{\omega}(F')}{m+\mu_{\omega}(F)} \le 1,
\end{align*}
by the $\omega$-semistability of $F$. 
Therefore (\ref{either1}) does not hold only 
if $\mu_{\omega}(F')=\mu_{\omega}(F)$
and $r \dim V'=r \dim V$. 

Suppose that (\ref{either1}) does not hold. 
Then by the above argument, we have 
$n'\dim V=n \dim V'$, hence it is enough to show that
\begin{align}\label{VA}
\dim V \cdot \dim A'>\dim V' \cdot \dim A. 
\end{align}
Since we have the surjection 
$V'\otimes K \otimes M\twoheadrightarrow F'\otimes \oO_{H}(m') \otimes M$
by the vanishing (\ref{vanish}), we have 
\begin{align*}
A'=\Imm \left(F'\otimes \oO_{H}(m') \otimes M
 \to F \otimes \oO_{H}(m') \otimes M \to A \right). 
\end{align*}
There is a commutative diagram, 
\begin{align}\label{COM}
\xymatrix{
0 \ar[r] & F' \otimes \oO_H \ar[r]\ar[d] & F\otimes \oO_H \ar[d] \\
0 \ar[r] & F'\otimes \oO_{H}(m') \otimes M \ar[r] \ar[d]^{\psi'} & 
F\otimes \oO_{H}(m') \otimes M \ar[d]^{\psi} \\
0 \ar[r] & A' \ar[r] & A.
}
\end{align}
By the assumption, 
$\Ker(\psi)$ is one dimensional, contained in $F\otimes \oO_H$, 
and
any non-zero element 
$s \in \Ker(\psi)$ gives a parabolic 
stable pair $(F, s)$. Since 
$\Ker(\psi') \subset \Ker(\psi)=\mathbb{C} \cdot s$, 
and the top square of (\ref{COM}) is Cartesian, 
the stability condition in Definition~\ref{defi:para}
implies that $\Ker(\psi')=0$. 
Hence $\dim A'=d_M h r'$, and together with 
$r\dim V'=r'\dim V$, $\dim A=d_M hr-1$, we obtain the 
inequality (\ref{VA}). 

(ii) $\Rightarrow$ (iii) : Obvious. 

(iii) $\Rightarrow$ (i) :
Suppose that $(\phi, \rho)$ is a GIT semistable 
point. First we show that 
$(\phi, \rho) \in \mathbb{P}(R_U) \subset R$. 
Since $\mathbb{P}(R_U)$ is projective over $U$, 
it is enough to show that $\rho \colon V\otimes \oO_X(-m) 
\to F$ is a point in $U$, i.e. $F$ is $\omega$-semistable 
and the map 
\begin{align}\label{map:V}
V \to H^0(X, F(m)),
\end{align}
 is an isomorphism. 
By the GIT stability, 
 for any proper subspace $V' \subset V$, 
the LHS of (\ref{V'V}) is positive for $l\gg 0$. 
By looking at the leading coefficients, we have 
\begin{align}\label{lead}
r' \dim V \ge r \dim V'. 
\end{align}
In order to show the map (\ref{map:V}) is an isomorphism, 
it is enough to show the injectivity of (\ref{map:V}). 
If (\ref{map:V}) is not injective, 
then there is a proper subspace $V' \subset V$ which 
generates a zero sheaf in $F$, 
which contradicts to (\ref{lead}). Therefore 
the map (\ref{map:V}) is an isomorphism. 

Next we show that $F$ is $\omega$-semistable. 
Let us take an exact sequence in $\Coh_{\le 1}(X)$, 
\begin{align}\label{F'F}
0 \to F' \to F \stackrel{\pi}{\to} F'' \to 0,
\end{align}
and suppose that $\mu_{\omega}(F')>\mu_{\omega}(F)$. 
Let $V'$ be the $\mathbb{C}$-vector space
$H^0(X, F'(m))$, which 
is considered to be a subspace of $V$
via the isomorphism (\ref{map:V}). 
Then by a choice of $m$ in Subsection~\ref{subsec:const}, 
$V'\otimes \oO_X(-m)$ generates $F'$. 
Applying (\ref{lead}),
we obtain $\mu_{\omega}(F) \ge \mu_{\omega}(F')$, 
a contradiction. Hence $F$ is $\omega$-semistable. 

Finally we show that $(\phi, \rho)$ determines a 
parabolic stable pair. Since $(\phi, \rho) \in \mathbb{P}(R_U)$, 
the surjection $\phi \colon V\otimes K \otimes M \twoheadrightarrow A$ factors
through $F\otimes \oO_{H}(m) \otimes M$, 
\begin{align*}
V\otimes K \otimes M \to F\otimes \oO_{H}(m') \otimes M
 \stackrel{\psi}{\to} A, 
\end{align*}
such that $\Ker(\psi)$ is one dimensional
and contained in $F\otimes \oO_H$. 
We would like to show that, for any non-zero element 
$s\in \Ker(\psi)$, the pair $(F, s)$ is a parabolic 
stable pair. 
Suppose by contradiction 
that there is an exact sequence of the form (\ref{F'F})
with $\mu_{\omega}(F')=\mu_{\omega}(F)=\mu_{\omega}(F'')$
and $(\pi \otimes \oO_H)(s) =0$. 
Setting $V'=H^0(X, F'(m))$, our choices of $m$ and $m'$
yield the commutative diagram, 
\begin{align*}
\xymatrix{
V'\otimes K \otimes M \ar[d] \ar[r] & F'\otimes \oO_H(m') \otimes M
 \ar[r]^{\qquad \quad \psi'} \ar[d]
& A' \ar[d] \\
V\otimes K \otimes M\ar[r] & F\otimes \oO_H(m') \otimes M \ar[r]^{\qquad \quad 
\psi}
& A.
}
\end{align*}
Here all the vertical arrows are injections and horizontal arrows are 
surjections. 
By the assumption $(\pi \otimes \oO_H)(s)=0$, 
we have $\Ker(\psi') =\Ker(\psi) =\mathbb{C} \cdot s$, 
hence $\dim A'=d_M hr'-1$. On the other hand, 
$(\rho, \phi)$ is 
GIT semistable by the assumption, hence the LHS of (\ref{V'V})
 should be non-negative. 
Also 
since 
$\mu_{\omega}(F)=\mu_{\omega}(F')$, 
we have 
\begin{align*}
\frac{r\dim V'}{r'\dim V}=\frac{m+\mu_{\omega}(F')}{m+\mu_{\omega}(F)}=1.
\end{align*}
Hence we have 
\begin{align}\label{eq1}
\frac{\dim V'}{\dim V}=\frac{r'}{r}=\frac{n'}{n}. 
\end{align}
Therefore we have 
\begin{align}\label{eq2}
\frac{\dim V' \cdot \dim A}{\dim V \cdot \dim A'}
=\frac{r'}{r} \cdot \frac{d_M hr-1}{d_M hr'-1}>1.
\end{align}
The equalities (\ref{eq1}), (\ref{eq2}) imply that 
the LHS of (\ref{V'V}) is negative, a contradiction. 
\end{proof}

\textit{Proof of Theorem~\ref{thm:moduli}:}
\begin{proof}
By Proposition~\ref{prop:GIT}, if we take 
$l\gg 0$, then we have 
\begin{align}\label{space}
R_{U}^{'\rm{s}}/\SL(V) =R^{\lL\mbox{-}\rm{s}}/\SL(V)
=R^{\lL\mbox{-}\rm{ss}}/\hspace{-.3em}/ \SL(V). 
\end{align}
Then by a general theory of GIT quotient~\cite{Mum}, 
the space (\ref{space}) 
is a projective scheme.
 By Lemma~\ref{lem:aut}, 
the automorphism group of any parabolic stable pair
is trivial, hence
there is a universal parabolic stable pairs on 
(\ref{space}). Hence
 the scheme (\ref{space})
is the desired fine moduli space. 
\end{proof}

\begin{rmk}\label{rmk:quasi}
If we take $H \in \lvert \oO_X(h) \rvert$ 
which does not satisfy the condition 
in Lemma~\ref{lem:trans}, then 
the moduli functor $\mM_n^{\rm{par}}(X, \beta)$
may not be represented by a projective scheme.
However in the argument of Theorem~\ref{thm:moduli}, let us replace  
$U \subset Q$ in (\ref{UQ}) by the open subscheme 
\begin{align*}
U^{\circ} \subset U \subset Q, 
\end{align*}
corresponding to quotients 
$V\otimes \oO_X(-m) \twoheadrightarrow F$
where $F$ intersects with $H$ transversally. 
Then following the same arguments, we can
show that $\mM_n^{\rm{par}}(X, \beta)$
is represented by a quasi-projective 
scheme over $\mathbb{C}$. 
\end{rmk}

\begin{rmk}
It is a natural question whether there is a symmetric perfect
obstruction theory on $M_n^{\rm{par}}(X, \beta)$ or not. 
This question seems to be not obvious
by the following reason. For instance 
in PT stable pair case~\cite{PT}, 
for a PT stable pair $(F, s)$ as in (\ref{PT/pair}), 
we have the associated object
in the derived category, 
$I^{\bullet} =(\oO_X \to F)$.
The perfect obstruction theory can be 
constructed by taking the cone of the trace morphism, 
\begin{align*}
\dR \Hom(I^{\bullet}, I^{\bullet}) \stackrel{\rm{tr}}{\to}
\dR \Hom(\oO_X, \oO_X). 
\end{align*}
In our case, by regarding a parabolic stable pair 
$(F, s)$ as a pair $N_{H/X}[-1]\stackrel{s}{\to} F$, 
we can associate $E \in \Coh(X)$ which fits into 
the exact sequence, 
\begin{align*}
0 \to F \to E \to N_{H/X} \to 0, 
\end{align*}
whose extension class is $s$. 
One might expect that, as an analogy of the trace map, 
there may be a natural morphism, 
\begin{align}\label{trace:N}
\dR \Hom(E, E) \to \dR \Hom(N_{H/X}, N_{H/S}), 
\end{align}
and taking its cone may give a perfect obstruction theory. 
Unfortunately there is no such a map (\ref{trace:N}), so 
we cannot discuss as in the PT stable pair case. 
\end{rmk}

\section{Wall-crossing formula}\label{sec:wall}
In this section, we introduce invariants 
counting parabolic stable pairs, and 
show that they are related to 
generalized DT invariants introduced in~\cite{JS}, \cite{K-S}. 
As in the previous section, $X$ is a smooth projective 
Calabi-Yau 3-fold over $\mathbb{C}$. 
\subsection{Counting invariants}\label{subsec:Count}
Let us take $d\in \mathbb{Z}_{>0}$
and a divisor 
$H \subset X$
satisfying the condition 
in Lemma~\ref{lem:trans}. By Theorem~\ref{thm:moduli}, 
for $n\in \mathbb{Z}$ and $\beta \in H_2(X, \mathbb{Z})$
with $\omega \cdot \beta \le d$, 
there is a fine moduli space 
$M_n^{\rm{par}}(X, \beta)$ which 
parameterizes parabolic stable pairs 
$(F, s)$ with $[F]=\beta$ and $\chi(F)=n$. 

Recall that, for any $\mathbb{C}$-scheme $M$,
Behrend~\cite{Beh} constructs
a canonical constructible function 
\begin{align*}
\nu \colon M \to \mathbb{Z},
\end{align*}
satisfying the following properties. 
\begin{itemize}
\item For $p\in M$, suppose that 
there is an analytic open neighborhood 
$p\in U$, a complex manifold $V$
and a holomorphic function $f\colon V\to \mathbb{C}$
such that $U\cong \{ df=0\}$. 
Then $\nu(p)$ is given by 
\begin{align*}
\nu(p)=(-1)^{\dim V}(1-\chi(M_p(f))). 
\end{align*}
Here $M_p(f)$ is the Milnor fiber of $f$ at $p$. 
\item If there is a symmetric perfect obstruction theory 
on $M$, we have 
\begin{align*}
\deg 
[M]^{\rm{vir}} &=\int_{M} \nu d\chi \\
&\cneq \sum_{m \in \mathbb{Z}} m \chi(\nu^{-1}(m)). 
\end{align*}
\end{itemize}
We define the invariant 
$\mathrm{DT}_{n, \beta}^{\rm{par}}$ as follows. 
\begin{defi}\label{defi:parainv}
We define $\mathrm{DT}_{n, \beta}^{\rm{par}} \in \mathbb{Z}$ to be
\begin{align*}
\mathrm{DT}_{n, \beta}^{\rm{par}} =
\int_{M_n^{\rm{par}}(X, \beta)}
\nu_M d \chi.
\end{align*}
Here $\nu_{M}$ is the Behrend function on 
$M_n^{\rm{par}}(X, \beta)$.  
\end{defi}
\begin{rmk}
We remark that the invariant $\mathrm{DT}_{n, \beta}^{\rm{par}}$
also depends on the choice of $\omega$ and $H$. 
We omit these notation in $\mathrm{DT}_{n, \beta}^{\rm{par}}$
for the simplification. 
\end{rmk}
In the local $(-1, -1)$-curve example, 
the above invariant can be computed very easily. 
\begin{exam}\label{exam:-1-1}
Let $f \colon X \to Y$, 
$C\subset X$ be as in Example~\ref{exam:-1}.
Suppose that there is $H \in \lvert \oO_X(1) \rvert$
which intersects with $C$ at a one point.  
Then by the classification of parabolic 
stable pairs in Example~\ref{exam:-1}, we have 
\begin{align*}
M_n^{\rm{par}}(X, m[C])
=\left\{ \begin{array}{cc}
\Spec \mathbb{C}, & m=1, \\
\emptyset, & m\ge 2.
\end{array}
   \right.
\end{align*}
Therefore we have 
\begin{align*}
\mathrm{DT}^{\rm{par}}_{n, m[C]}
=\left\{ \begin{array}{cc}
1, & m=1, \\
0, & m\ge 2.
\end{array}
   \right.
\end{align*}
\end{exam}
We introduce the generating series of 
$\mathrm{DT}_{n, \beta}^{\rm{par}}$ as follows. 
For $\mu \in \mathbb{Q}$, we set
\begin{align}\label{para:series}
\mathrm{DT}^{\rm{par}}(\mu, d) \cneq 
1+
\sum_{\begin{subarray}{c} 0<\beta \cdot \omega \le d \\
n/\omega \cdot \beta =\mu
\end{subarray}}
\DT_{n, \beta}^{\rm{par}}q^n t^{\beta}. 
\end{align}
The above series is contained in the ring $\Lambda_{\le d}$ 
defined as follows. First the ring $\Lambda$
is defined by 
\begin{align*}
\Lambda \cneq \bigoplus_{n \in \mathbb{Z}, \beta>0}
\mathbb{Q} q^n t^{\beta}. 
\end{align*}
Here $\beta>0$ means that $\beta$ is a numerical class of 
an effective one cycle on $X$. 
The ring $\Lambda$ is defined by the quotient ring 
of $\Lambda$ by the ideal generated by 
$q^n t^{\beta}$ with $\omega \cdot \beta >d$. 
We have 
\begin{align*}
\mathrm{DT}^{\rm{par}}(\mu, d) \in \Lambda_{\le d}.
\end{align*}
\subsection{Category of parabolic pairs}\label{subsec:category}
Let $(F, s)$ be a parabolic stable pair as in Definition~\ref{defi:para}. 
As we observed in Subsection~\ref{subsec:relation}, 
the pair $(F, s)$ can be also interpreted as a pair, 
\begin{align}\label{pair:cate}
N_{H/X}[-1] \stackrel{s}{\to} F. 
\end{align}
In this subsection and next subsection, 
we construct the category of pairs as above, 
and interpret the stability condition in Definition~\ref{defi:para}
in terms of the pair (\ref{pair:cate}). 

Let $(X, \omega, d, H)$ be as in Lemma~\ref{lem:trans}. 
For $\mu \in \mathbb{Q}$, 
the category $\aA(\mu, d)$ is defined as follows. 
\begin{defi}
We define the category $\aA(\mu, d)$ to be the category of 
pairs, 
\begin{align*}
N_{H/X}^{\oplus r}[-1] \stackrel{s}{\to} F, 
\end{align*}
where $r\in \mathbb{Z}_{\ge 0}$
and $F$ is a one dimensional $\omega$-semistable
sheaf satisfying
\begin{align*}
\mu_{\omega}(F)=\mu, \quad \omega \cdot [F] \le d. 
\end{align*}
For two objects
 $E_i =(N_{H/X}^{\oplus r_i}[-1] \stackrel{s_i}{\to} F_i) \in \aA(\mu, d)$
with $i=1, 2$, the set of morphisms
$\Hom(E_1, E_2)$ is given by 
the commutative diagram, 
\begin{align}\label{mor:para}
\xymatrix{
N_{H/X}^{\oplus r_1}[-1] \ar[r]^{s_1} \ar[d]_{\phi \otimes 
\mathrm{id}_{N_{H/X}}} 
& F_1  \ar[d]^{g} \\
N_{H/X}^{\oplus r_2}[-1] \ar[r]^{s_2} & F_2,
}
\end{align}
for $\phi \in M(r_1, r_2)$ and $g \in \Hom(F_1, F_2)$. 
\end{defi}
Let $g \colon F_1 \to F_2$ be a 
morphism of one dimensional $\omega$-semistable 
sheaves with $\mu_{\omega}(F_i)=\mu$. 
It is easy to see that $\Ker(g)$, $\Imm(g)$ and 
$\Cok(g)$ are all $\omega$-semistable 
sheaves with $\mu_{\omega}(\ast)=\mu$. 
Also let us take an exact sequence of $\omega$-semistable
one dimensional sheaves,
\begin{align*}
0 \to F_1 \to F_2 \to F_3 \to 0.
\end{align*}
If $\omega \cdot [F_2] \le d$, then 
we have the 
exact sequence by Lemma~\ref{lem:Ext},
\begin{align}\label{Ex:Ext}
0 \to \Ext_{X}^1(N_{H/X}, F_1) \to \Ext_{X}^1(N_{H/X}, F_2) \to 
\Ext_X^1(N_{H/X}, F_3) \to 0.
\end{align}
The condition $\omega \cdot [F_2] \le d$
is required for the divisor $H$ to intersect 
with $F_i$ transversally. 
Hence for a morphism (\ref{mor:para}), 
we can define its kernel, image and cokernel
in $\aA(\mu, d)$. 
For instance, the kernel is given by 
\begin{align*}
\Ker(\phi) \otimes N_{H/X}[-1] \to \Ker(g). 
\end{align*}
The notion of monomorphisms, 
epimorphisms and 
exact sequences in $\aA(\mu, d)$
can be defined in an usual way
using the above kernel, image and the cokernel. 

However for two objects
\begin{align}\label{two:obj}
E_i=(N_{H/X}^{\oplus r_i}[-1] \stackrel{s_i}{\to} F_i) \in \aA(\mu, d), 
\end{align}
with $i=1, 2$,  
the set of extensions
in $\aA(\mu, d)$,
\begin{align}\label{par:ext}
0 \to E_1 
\to (N_{H/X}^{\oplus r}[-1] \stackrel{s}{\to} F)
\to E_2 \to 0,
\end{align}
can be defined \textit{only if} 
the following condition holds:
\begin{align}\label{restrict}
\omega \cdot ([F_1]+[F_2]) \le d.
\end{align}
In particular $E_1 \oplus E_2$ cannot be defined 
without the condition (\ref{restrict}).
This implies that $\aA(\mu, d)$ is not an abelian category. 

For our purpose, 
 we only use extensions (\ref{par:ext}) 
satisfying the condition (\ref{restrict}). 
Except the above restriction for the 
possible extensions, the category $\aA(\mu, d)$
behaves as if it is a $\mathbb{C}$-linear abelian category. 
For instance if the condition (\ref{restrict}) is satisfied, 
then the set of isomorphism 
classes of extensions (\ref{par:ext}), 
$\Ext^1(E_2, E_1)$, is a finite dimensional 
$\mathbb{C}$-vector space.
Furthermore the following lemma holds. 
\begin{lem}\label{lem:exseq}
For two objects (\ref{two:obj}), 
suppose that the condition (\ref{restrict}) holds. 
Then we have the 
following exact sequence of 
$\mathbb{C}$-vector spaces, 
\begin{align}\label{exact:HME}
0 &\to \Hom(E_2, E_1) \to 
M(r_2, r_1) \oplus \Hom(F_2, F_1) \\
& \to \Hom(N_{H/X}^{\oplus r_2}[-1], F_1)
\to \Ext^1(E_2, E_1) \to \Ext^1(F_2, F_1) \to 0. 
\end{align}
\end{lem}
\begin{proof}
Let $E=(N_{H/X}^{\oplus r}[-1] \stackrel{s}{\to} F)$ be 
an object in $\aA(\mu, d)$ which fits into 
an exact sequence (\ref{par:ext}).
Then we have the exact sequence of sheaves,
\begin{align}\label{F1FF2}
0\to F_1 \to F \to F_2 \to 0, 
\end{align} 
hence we obtain a linear map, 
\begin{align*}
\gamma_1 \colon \Ext^1(E_2, E_1) \to \Ext^1(F_2, F_1), 
\end{align*}
sending $E$ to $F$. 
The map $\gamma_1$ is surjective by the exact sequence (\ref{Ex:Ext}) 
applied to (\ref{F1FF2}). 

Let us look at the kernel of $\gamma_1$. 
The kernel of $\gamma_1$ consists of 
isomorphism classes of extensions in $\mathrsfs{A}(\mu, d)$
of the form
\begin{align*}
\xymatrix{
0 \ar[r] & N_{H/X}^{\oplus r_1}[-1]  \ar[r]^{i_1} \ar[d]^{s_1} 
& N_{H/X}^{\oplus r_1 + r_2}[-1] \ar[r]^{i_2}
 \ar[d]^{s} & N_{H/X}^{\oplus r_2}[-1]
\ar[r] \ar[d]^{s_2} & 0, \\
0 \ar[r] & F_1 \ar[r]^{j_1} & F_1 \oplus F_2 \ar[r]^{j_2} & F_2 \ar[r] & 0.
}
\end{align*}
Here $i_1, j_1$ are 
embedding into corresponding factors, 
 and $i_2, j_2$ are projections onto 
corresponding factors.  
The above extension 
is given if we give a $\Hom(N_{H/X}^{\oplus r_2}[-1], F_1)$-factor 
of $s$, hence we obtain a surjection, 
\begin{align*}
\gamma_2 \colon \Hom(N_{H/X}^{\oplus r_2}[-1], F_1) \twoheadrightarrow 
\Ker(\gamma_1). 
\end{align*}
Next we look at the kernel of $\gamma_2$. 
An element $s' \in \Hom(N_{H/X}^{\oplus r_2}[-1], F_1)$
is contained in $\Ker(\gamma_2)$ if and only if
there are split projections 
$i_1', j_1'$ of $i_1, j_1$ respectively 
such that the following diagram commutes, 
\begin{align}\label{diag:N}
\xymatrix{
N_{H/X}^{\oplus r_1 +r_2}[-1] \ar[r]^{i_1'} 
\ar[d]_{s}
  & 
N_{H/X}^{\oplus r_1}[-1] \ar[d]^{s_1}, \\
F_1 \oplus F_2 \ar[r]^{j_1'} & F_1,
}
\end{align}
where $s$ is given by the matrix, 
\begin{align*}
s= \left(\begin{array}{cc}
s_1 & s' \\
0 & s_2
\end{array}
\right). 
\end{align*}
Let $\phi$ and $\psi$ be 
$\Hom(N_{H/X}^{\oplus r_2}[-1], N_{H/X}^{\oplus r_1}[-1])$
and $\Hom(F_2, F_1)$-components of 
$i_1'$ and $j_1'$ respectively. Then the diagram (\ref{diag:N}) commutes
if and only if the following equality holds:
\begin{align*}
s'=s_1 \circ \phi -\psi \circ s_2. 
\end{align*}
Hence we obtain the surjection, 
\begin{align*}
\gamma_3 \colon 
M(r_2, r_1) \oplus \Hom(F_2, F_1) \twoheadrightarrow \Ker(\gamma_2), 
\end{align*}
sending $(\phi, \psi)$
to $s_1 \circ \phi- \psi \circ s_2$. 

Finally the kernel of $\gamma_3$
coincides with $\Hom(E_2, E_1)$
by its definition. Therefore we obtain the 
exact sequence (\ref{exact:HME}). 
\end{proof}

\subsection{Weak stability conditions on $\aA(\mu, d)$}
In this subsection, we construct weak stability 
conditions on $\aA(\mu, d)$ and investigate their 
relationship to parabolic stability. 
First we define the slope function on $\aA(\mu, d)$. 
\begin{defi}
For a non-zero
 object $E=(N_{H/X}^{\oplus r}[-1] \stackrel{s}{\to} F) \in \aA(\mu, d)$
and $\alpha \in \mathbb{Q}$, 
we set $\widehat{\mu}_{\alpha}(E)$ to be 
\begin{align*}
\widehat{\mu}_{\alpha}(E) =
\left\{ \begin{array}{cc}
\alpha, & \mbox{ if } r\neq 0, \\
\mu (=\mu_{\omega}(F)), & \mbox{ if } r=0. 
\end{array}  \right. 
\end{align*}
\end{defi}
The above $\widehat{\mu}_{\alpha}$-slope function
satisfies the following weak seesaw property. 
Let $0 \to E_1 \to E_2 \to E_3 \to 0$
be an exact sequence in $\aA(\mu, d)$
with $E_1, E_3 \neq 0$
in the sense explained in the previous subsection. 
Then either one of the following 
conditions hold:
\begin{align*}
\widehat{\mu}_{\alpha}(E_1) \ge \widehat{\mu}_{\alpha}(E_2) \ge \widehat{\mu}_{\alpha}(E_3), \\
\widehat{\mu}_{\alpha}(E_1) \le \widehat{\mu}_{\alpha}(E_2) \le \widehat{\mu}_{\alpha}(E_3).
\end{align*}
The $\widehat{\mu}_{\alpha}$-stability on $\aA(\mu, d)$
is defined as follows. 
\begin{defi}
An object $E \in \aA(\mu, d)$ is 
$\widehat{\mu}_{\alpha}$-(semi)stable if for any 
exact sequence $0 \to E_1 \to E_2 \to E_3 \to 0$
in $\aA(\mu, d)$ with $E_1, E_3 \neq 0$, 
we have the inequality, 
\begin{align*}
\widehat{\mu}_{\alpha}(E_1)<(\le) \widehat{\mu}_{\alpha}(E_3). 
\end{align*}
\end{defi}
The $\mu_{\alpha}$-stability satisfies the usual 
Harder-Narasimhan (HN) property. 
\begin{lem}\label{lem:NH}
For any $E \in \aA(\mu, d)$, there is a filtration in $\aA(\mu, d)$, 
\begin{align}\label{HN}
0=E_0 \subset E_1 \subset \cdots \subset E_N=E,  
\end{align}
with $N\le 2$
such that each $F_i=E_i/E_{i-1}$ is $\widehat{\mu}_{\alpha}$-semistable 
satisfying 
$\widehat{\mu}_{\alpha}(F_i)>\widehat{\mu}_{\alpha}(F_{i+1})$
for all $i$.  
The exact sequence (\ref{HN}) is unique up to isomorphism. 
\end{lem}
\begin{proof}
Note that for any non-zero object $(N_{H/X}^{\oplus r}[-1] \to F) \in 
\aA(\mu, d)$, we have $r\ge 0$, $\omega \cdot [F] \ge 0$
and either one of them is non-zero. 
Hence there are no infinite sequences in $\aA(\mu, d)$, 
\begin{align*}
E&=E_0 \supset E_1 \supset \cdots \supset E_n \supset \cdots, \\
E&=E_0 \twoheadrightarrow E_1 \twoheadrightarrow \cdots \twoheadrightarrow
E_n \twoheadrightarrow \cdots. 
\end{align*}
Therefore the criterion in~\cite[Proposition~2.12]{Tcurve1} is satisfied, 
and there are HN filtrations 
with respect to $\widehat{\mu}_{\alpha}$-stability. 
Since $\widehat{\mu}_{\alpha}(\ast) \in \{\alpha, \mu\}$, 
the HN filtrations are at most 2-steps, i.e. $N\le 2$. 
\end{proof}
Next we see the relationship between 
$\widehat{\mu}_{\alpha}$-stability and parabolic stability. 
\begin{prop}\label{prop:stab}
For an object 
\begin{align}\label{Eobj}
E=(N_{H/X}[-1] \to F) \in \aA(\mu, d),
\end{align}
we have the following. 

(i) If $\alpha<\mu$, then $E$ is $\widehat{\mu}_{\alpha}$-semistable 
if and only if $F=0$. 

(ii) If $\alpha=\mu$, then any object $E$ as in (\ref{Eobj})
is $\widehat{\mu}_{\alpha}$-semistable. 

(iii) If $\alpha>\mu$, then $E$ is $\widehat{\mu}_{\alpha}$-semistable 
if and only if it is $\widehat{\mu}_{\alpha}$-stable, if and only if 
the pair  
\begin{align}\label{Fs}
(F, s), \quad s\in F\otimes \oO_H, 
\end{align}
determined by (\ref{Eobj})
and Lemma~\ref{lem:Ext} is a parabolic stable pair. 
\end{prop}
\begin{proof}
(i) Suppose that $E$ is $\widehat{\mu}_{\alpha}$-semistable 
and $F\neq 0$. Then we have the 
exact sequence in $\aA(\mu, d)$, 
\begin{align*}
0 \to (0\to F) \to E \to (N_{H/X}[-1] \to 0) \to 0. 
\end{align*}
Since we have 
\begin{align*}
\widehat{\mu}_{\alpha}(0\to F)=\mu, \quad
\widehat{\mu}_{\alpha}(N_{H/X}[-1] \to F)=\alpha,
\end{align*}
the above sequence destabilize $E$, a contradiction. 

Conversely if $E=(N_{H/X}[-1] \to 0)$, then 
there is no exact sequence 
$0\to E_1 \to E \to E_2 \to 0$ with 
$E_1, E_2 \neq 0$. 
In particular $E$ is $\widehat{\mu}_{\alpha}$-stable. 

(ii) Obvious. 

(iii) Suppose that $E$ is $\widehat{\mu}_{\alpha}$-semistable 
and assume that there is a surjection 
$F \stackrel{\pi}{\twoheadrightarrow} F'$ with $\mu_{\omega}(F')=\mu$
such that $(\pi \otimes \oO_H)(s)=0$. Then by Lemma~\ref{lem:Ext}, 
there is an exact sequence in $\aA(\mu, d)$ of the form, 
\begin{align}\label{F'''}
0 \to (N_{H/X}[-1] \to F'') \to E \to (0 \to F') \to 0.
\end{align}
Since $\alpha>\mu$, the sequence (\ref{F'''})
destabilizes $E$, a contradiction. 

Conversely suppose that the pair (\ref{Fs})
is parabolic stable and take an exact sequence 
in $\aA(\mu, d)$, 
\begin{align*}
0\to (N_{H/X}^{\oplus a}[-1] \to F')
 \to E \to (N_{H/X}^{\oplus b}[-1] \to F'') \to 0. 
\end{align*}
Since $a+b=1$, we have the two 
possibilities, $(a, b)=(1, 0)$ and $(a, b)=(0, 1)$. 
When $(a, b)=(1, 0)$, then 
the surjection $F \to F''$ takes $s$ to $0$
by taking $\otimes \oO_H$, hence it 
contradicts to the parabolic stability. 
When $(a, b)=(0, 1)$, then 
we have 
\begin{align*}
\widehat{\mu}_{\alpha}(0 \to F') 
=\mu <\alpha =\widehat{\mu}_{\alpha}(N_{H/X}[-1] \to F''). 
\end{align*}
Hence $E$ is $\widehat{\mu}_{\alpha}$-stable. 
\end{proof}
As a corollary, we have the following. 
\begin{cor}\label{cor:later}
Let $(F, s)$ be a parabolic stable pair 
and $E=(N_{H/X}[-1] \to F)$ the associated 
object in $\aA(\mu, d)$ by Lemma~\ref{lem:Ext}.
Then we have 
\begin{align}\label{HomEE}
\Hom(E, E) =\mathbb{C}. 
\end{align} 
In particular the group $\Aut(F, s)$
defined in Lemma~\ref{lem:aut} is $\{\mathrm{id}_F\}$.  
\end{cor}

\begin{proof}
By Proposition~\ref{prop:stab}, the object $E$ is
$\widehat{\mu}_{\alpha}$-stable, hence 
 (\ref{HomEE}) follows from 
a general argument of stable 
objects. (cf.~\cite[Corollary~1.2.8]{Hu}.) 
The group $\Aut(F, s)$ is identified with 
the subset of $\Hom(E, E)$, consisting of 
commutative diagrams, 
\begin{align}\label{mor:para2}
\xymatrix{
N_{H/X}[-1]^{} \ar[r]^{} \ar[d]_{\mathrm{id}} & F 
 \ar[d]^{g} \\
N_{H/X}[-1]^{} \ar[r]^{} & F,
}
\end{align}
where $g$ is an isomorphism. Hence by (\ref{HomEE}), 
$g$ must be an identity. 
\end{proof}

As another corollary, we can 
explicitly describe the HN 
filtrations of objects of 
the form $(N_{H/X}[-1] \to F)$.

\begin{cor}\label{cor:HN}
For an object $E=(N_{H/X}^{}[-1] \to F) \in \aA(\mu, d)$, 
we have the following. 

(i) If $\alpha<\mu$,
then the HN filtration of $E$ is
either $0=E_0 \subset E_1=E$ or 
 the exact sequence, 
\begin{align*}
0 \to (0 \to F) \to E \to (N_{H/X}[-1] \to 0) \to 0. 
\end{align*} 

(ii) If $\alpha=\mu$, then the HN filtration 
of $E$ is $0=E_0 \subset E_1=E$. 

(iii) If $\alpha>\mu$, then the HN
filtration of $E$ is 
either $0=E_0 \subset E_1 =E$ or 
an exact sequence of the form, 
\begin{align*}
0 \to (N_{H/X}[-1] \stackrel{s'}{\to} F') \to E \to 
(0 \to F'') \to 0, 
\end{align*}
where $(N_{H/X}[-1] \stackrel{s'}{\to}F')$
is determined by a parabolic stable pair $(F', s')$.
\end{cor}
\begin{proof}
The result is obvious from Lemma~\ref{lem:NH} and 
Proposition~\ref{prop:stab}. 
\end{proof}

\subsection{Stack of objects in $\aA(\mu, d)$}\label{subsec:stackA}
In this subsection, we study stack of objects 
in the category $\aA(\mu, d)$. For the introduction to
stacks, see~\cite{GL}.

Let $\mathrsfs{A}(\mu, d)$ be the 2-functor, 
\begin{align*}
\mathrsfs{A}(\mu, d) \colon 
\mathrm{Sch}/\mathbb{C} \to \mathrm{groupoid},
\end{align*}
sending an $\mathbb{C}$-scheme $T$ to the groupoid of 
pairs, 
\begin{align}\label{pair:fami}
\nN[-1] \to \fF, 
\end{align}
where 
$\nN$ and $\fF$
are $T$-flat coherent sheaves on $X\times T$, 
such that for any closed point $t\in T$, 
there is a commutative diagram, 
\begin{align*}
\xymatrix{
\nN_t[-1] \ar[r] \ar[d]^{\psi} & \fF_t \ar[d]^{h} \\
N_{H/X}^{\oplus r}[-1] \ar[r] & F.
}
\end{align*}
Here the bottom pair is an object in $\aA(\mu, d)$, 
and $\psi$, $h$ are isomorphisms of sheaves on $X$. 

The morphisms in the groupoid $\mathrsfs{A}(\mu, d)(T)$ 
are given by  
the commutative diagrams, 
\begin{align}\label{mor:para2}
\xymatrix{
\nN_1[-1] \ar[r]^{} \ar[d]_{\phi} & \fF_1 
 \ar[d]^{g} \\
\nN_2[-1]  \ar[r]^{} & \fF_2,
}
\end{align}
where $\phi$ and $g$ are isomorphisms of sheaves on $X\times T$. 

The stack $\mathrsfs{A}(\mu, d)$ can be easily 
 shown to be a global quotient stack 
of some scheme locally of finite type
over $\mathbb{C}$, 
in particular, 
it is an Artin stack of finite type 
over $\mathbb{C}$. 
In order to see this, 
 we decompose $\mathrsfs{A}(\mu, d)$ into 
components, 
\begin{align*}
\mathrsfs{A}(\mu, d)=\coprod_{\begin{subarray}{c}
(r, \beta, n) \in \Gamma(\mu), \\
r\ge 0, \ 0\le \omega \cdot \beta \le d
\end{subarray}}
\mathrsfs{A}_{r, \beta, n}, 
\end{align*}
where $\mathrsfs{A}_{r, \beta, n}$ is the stack 
of pairs $N_{H/X}^{\oplus r}[-1] \to F$
with $[F]=\beta$ and $\chi(F)=n$,
and $\Gamma(\mu)$ is the abelian group 
defined by 
\begin{align}\label{Gmu}
\Gamma(\mu)\cneq \{ (r, \beta, n) \in \mathbb{Z} \oplus 
H_2(X, \mathbb{Z}) \oplus \mathbb{Z} : n=\mu(\omega \cdot \beta)\}. 
\end{align}
Below we use the notation in Subsection~\ref{subsec:const}. 
For fixed $\beta$ and $n$ as above, we 
take $m\gg 0$
and the Quot scheme $Q$
as in (\ref{Quot}).
Let $U\subset Q$ be the open subscheme as in (\ref{UQ}).
We have the following coherent sheaf on $U$, 
\begin{align}\label{RUr}
R_{U}^{(r)} \cneq 
\eE xt_{\pi_U}^1(N_{H/X}^{\oplus r} \boxtimes \oO_U, \mathbb{F}_{U}), 
\end{align} 
where $\eE xt_{\pi_U}^i(\ast, \ast)$ is the $i$-th
derived functor 
with respect to the functor $\pi_{U\ast}\hH om(\ast, \ast)$. 
By Lemma~\ref{lem:Ext}, 
the sheaf $R_{U}^{(r)}$ is a vector bundle and 
we have the canonical isomorphism, 
\begin{align*}
R_U^{(r)} \cong 
\pi_{U\ast}(\mathbb{F}_{U} \otimes \oO_{H\times U}^{\oplus r}). 
\end{align*}
In particular $R_{U}^{(1)}$ coincides with 
$R_U$ given in (\ref{RU}). 
We also denote the total space of the 
vector bundle (\ref{RUr}) as $R_U^{(r)}$. 

The space $R_{U}^{(r)}$ parameterizes
data, 
\begin{align*}
V \otimes \oO_X(-m) \twoheadrightarrow F
\leftarrow N_{H/X}^{\oplus r}[-1], 
\end{align*}
where the left arrow represents 
a point in $U$. 
The groups $\GL(V)$, 
$\GL(r, \mathbb{C})$ act on 
$V\otimes \oO_X(-m)$, $N_{H/X}^{\oplus r}[-1]$
respectively. Hence 
we obtain the action of 
$\GL(V) \times \GL(r, \mathbb{C})$ on 
$R_{U}^{(r)}$. 
The stack $\mathrsfs{A}_{r, \beta, n}$ is constructed to be 
the quotient stack, 
\begin{align}\label{Astack}
\mathrsfs{A}_{r, \beta, n} =
\left[R_U^{(r)}/\left(\GL(V) \times \GL(r, \mathbb{C}) \right)   \right]. 
\end{align}

\subsection{Hall algebra}
In this subsection, we recall
the Hall type algebra of 
the category $\aA(r, d)$
based on the work~\cite{Joy2}, 
using the Artin stack $\mathrsfs{A}(\mu, d)$. 

The $\mathbb{Q}$-vector space
$H(\mu, d)$ is spanned by 
the isomorphism classes of pairs, 
\begin{align}\label{iso:pair}
(\mathrsfs{X}, \rho), 
\end{align}
where 
$\mathrsfs{X}$ is an Artin stack of finite 
type over $\mathbb{C}$
with affine stabilizers, and $\rho$
is a 1-morphism, 
\begin{align*}
\rho \colon \mathrsfs{X} \to \mathrsfs{A}(\mu, d). 
\end{align*}
Two pairs $(\mathrsfs{X}_i, \rho_i)$
for $i=1, 2$ are isomorphic if 
there is an 1-isomorphism $f \colon \mathrsfs{X}_1 
\to \mathrsfs{X}_2$ which 2-commutes
with $\rho_i$. 
The $\mathbb{Q}$-vector space $H(\mu, d)$
is decomposed as 
\begin{align*}
H(\mu, d)=\bigoplus_{\begin{subarray}{c}
(r, \beta, n) \in \Gamma(\mu) \\
r\ge 0, \ 0\le \omega \cdot \beta \le d
\end{subarray}}
H_{r, \beta, n},
\end{align*}
where $H_{r, \beta, n}$ is the subvector 
space of $H(\mu, d)$, spanned by 
(\ref{iso:pair}) 
such that $\rho$ factors through 
the substack $\mathrsfs{A}_{r, \beta, n} \subset \mathrsfs{A}(\mu, d)$. 

Let $\mathrsfs{E} x(\mu, d)$ be the 
stack of exact sequences in $\aA(\mu, d)$, 
\begin{align}\label{E123}
0 \to E_1 \to E_2 \to E_3 \to 0. 
\end{align}
As in the previous subsection, it is 
not difficult to show that 
$\mathrsfs{E} x(\mu, d)$ is 
an Artin stack locally of finite type 
over $\mathbb{C}$, and we omit the detail. 
We have the 1-morphisms, 
\begin{align*}
p_i \colon \mathrsfs{E} x(\mu, d)
\to \mathrsfs{A}(\mu, d), 
\end{align*}
sending an exact sequence (\ref{E123})
to the object $E_i$. 

For two elements
 $\alpha_i=(\mathrsfs{X}_i, \rho_i) \in H(\mu, d)$
with $i=1, 2$, its $\ast$-product
$\alpha_1 \ast \alpha_2$ is defined in
the following way. 
We have the diagram, 
\begin{align*}
\xymatrix{
\mathrsfs{Y} \ar[r]^{\rho} \ar[d] & 
\mathrsfs{X} \ar[r]^{p_2} \ar[d]^{(p_1, p_3)} & \mathrsfs{A}(\mu, d), \\
\mathrsfs{X}_1 \times \mathrsfs{X}_2 \ar[r]^{(\rho_1, \rho_2)} & 
\mathrsfs{A}(\mu, d)^{\times 2}. &
}
\end{align*}
Here the left diagram is a Cartesian square. 
The product $\alpha_1 \ast \alpha_2$ is given by 
\begin{align*}
\alpha_1 \ast \alpha_2 =
[(\mathrsfs{Y}, p_2 \circ \rho)] \in H(\mu, d). 
\end{align*}
We have the following proposition. 
\begin{prop}
The $\ast$-product on $H(\mu, d)$
is an associative product on $H(\mu, d)$, 
with unit given by 
$[\Spec \mathbb{C} \to \mathrsfs{A}(\mu, d)]$, 
corresponding to $0\in \aA(\mu, d)$. 
\end{prop}
\begin{proof}
Although the category $\aA(\mu, d)$ is not an 
abelian category, the same proof 
as in~\cite[Theorem~5.2]{Joy2} works. 
The only thing to notice is that, for
 $\alpha_i \in H_{r_i, \beta_i, n_i}$
with $i=1, 2, 3$ and 
\begin{align*}
\omega \cdot (\beta_1 + \beta_2 +\beta_3)>d,
\end{align*}
we have the equalities, 
\begin{align*}
\alpha_1 \ast (\alpha_2 \ast \alpha_3)=
(\alpha_1 \ast \alpha_2) \ast \alpha_3 =0.
\end{align*}
\end{proof}
\subsection{Elements $\delta_{r, \beta, n}^{\alpha}$, 
$\epsilon_{r, \beta, n}^{\alpha}$}
For $\alpha \in \mathbb{Q}$, we have the substack, 
\begin{align*}
\mathrsfs{M}^{\alpha}_{r, \beta, n} \subset \mathrsfs{A}_{r, \beta, n}, 
\end{align*}
parameterizing $\widehat{\mu}_{\alpha}$-semistable objects
$(N_{H/X}^{\oplus r}[-1] \to F) \in \aA(\mu, d)$ with $[F]=\beta$, 
$\chi(F)=n$. 
 
It is not difficult to check that the stack 
$\mathrsfs{M}^{\alpha}_{r, \beta, n}$
is an open finite type substack of $\mathrsfs{A}_{r, \beta, n}$, 
hence it is an Artin stack of finite type over $\mathbb{C}$. 
We will only need this fact for $r=0, 1$, the cases which 
follow from Proposition~\ref{prop:stab} and Theorem~\ref{thm:moduli}
immediately. 
For instance, suppose that $r=1$ and $\alpha>\mu$. 
Then in the notation of the previous section, 
we have 
\begin{align}\label{Mstack}
\mathrsfs{M}^{\alpha}_{1, \beta, n}
=\left[ R_{U}^{\mathrm{s}} / \left( \mathbb{C}^{\ast} \times 
\mathrm{GL}(V) \right)    \right] 
\end{align}
by (\ref{Astack}) and Proposition~\ref{prop:stab} (iii).
By Proposition~\ref{prop:GIT}, the RHS of (\ref{Mstack})
is an open substack of $\mathrsfs{A}_{r, \beta, n}$, and 
we have 
\begin{align}\label{one:isom}
\mathrsfs{M}_{1, \beta, n}^{\alpha} 
\cong \left[M_{n}^{\rm{par}}(X, \beta)/\mathbb{C}^{\ast} \right].
\end{align}
Here $\mathbb{C}^{\ast}$ acts on $M_n^{\rm{par}}(X, \beta)$ trivially. 

\begin{rmk}\label{donot:depend}
(i) For any $\alpha$, the $\mathbb{C}$-valued point of 
the stack 
$\mathrsfs{M}_{r, 0, 0}^{\alpha}$ consists of 
$(N_{H/X}^{\oplus r}[-1] \to 0)$, hence 
\begin{align}\label{Mr00}
\mathrsfs{M}_{r, 0, 0} \cong \left[ \Spec \mathbb{C}/\GL(r, \mathbb{C})
   \right]. 
\end{align}
In particular the stack (\ref{Mr00}) does not depend on $\alpha$. 

(ii) For any $\alpha$, the $\mathbb{C}$-valued points 
of the stack 
$\mathrsfs{M}_{0, \beta, n}^{\alpha}$ consist of the objects
of the form $(0 \to F)$. 
In particular the stack $\mathrsfs{M}_{0, \beta, n}^{\alpha}$
does not depend on $\alpha$, and isomorphic to 
the moduli stack of one dimensional $\mu_{\omega}$-semistable sheaves 
$F$ on $X$ with $[F]=\beta$, $\chi(F)=n$. 
\end{rmk}

Following~\cite[Definition~3.1.8]{Joy4}, we 
define $\delta_{r, \beta, n}^{\alpha}$
and $\epsilon_{r, \beta, n}^{\alpha} \in H_{r, \beta, n}$ to be 
\begin{align}\label{delta}
\delta_{r, \beta, n}^{\alpha}
&=[\mathrsfs{M}_{r, \beta, n}^{\alpha} \hookrightarrow 
\mathrsfs{A}_{r, \beta, n}], \\
\epsilon_{r, \beta, n}^{\alpha}
&= \sum_{\begin{subarray}{c}
l\ge 1, 
(r_i, \beta_i, n_i) \in \Gamma(\mu), 1\le i\le l, 
\\
\label{epsilon}
(r_1, \beta_1, n_1)+ \cdots +(r_l, \beta_l, n_l)
=(r, \beta, n), \\
\widehat{\mu}_{\alpha}(r_i, \beta_i, n_i)=\widehat{\mu}_{\alpha}(r, \beta, n).
\end{subarray}}
\frac{(-1)^{l-1}}{l} \delta_{r_1, \beta_1, n_1}^{\alpha} \ast 
\cdots \ast \delta_{r_l, \beta_l, n_l}^{\alpha}. 
\end{align}
Here $\widehat{\mu}_{\alpha}(r, \beta, n)$ is
defined to be 
$\alpha$ if $r\neq 0$ and $\mu$
if $r=0$. 
In other words,
the element $\epsilon_{r, \beta, n}^{\alpha}$
is given by  
\begin{align*}
\sum_{\widehat{\mu}_{\alpha}(r, \beta, n)=\mu'}
\epsilon_{r, \beta, n}^{\alpha}
= \log \left( 1 + \sum_{\widehat{\mu}_{\alpha}(r, \beta, n)=\mu'}
\delta_{r, \beta, n}^{\alpha}  \right),
\end{align*}
for $\mu' \in \{\alpha, \mu\}$. Or equivalently, 
\begin{align}\label{Or:equiv}
 1 + \sum_{\widehat{\mu}_{\alpha}(r, \beta, n)=\mu'}
\delta_{r, \beta, n}^{\alpha} = \exp 
\left(\sum_{\widehat{\mu}_{\alpha}(r, \beta, n)=\mu'}
\epsilon_{r, \beta, n}^{\alpha}   \right).
\end{align}
Note that, by Remark~\ref{donot:depend}, the elements
$\delta_{r, 0, 0}^{\alpha}$, $\delta_{0, \beta, n}^{\alpha}$, 
$\epsilon_{0, \beta, n}^{\alpha}$ do not 
depend on $\alpha$. Below we omit $\alpha$ in the notation
for these elements.
Also for the convention, we set $\delta_{0, 0, 0}=1$. 

Using the results of Corollary~\ref{cor:HN}, 
the following lemma obviously follows. 
\begin{prop}\label{prop:ident}
Take $\alpha>\mu$
and $(0, \beta, n) \in \Gamma(\mu)$. 
Then we have the following identity in $H(\mu, d)$, 
\begin{align*}
\delta_{1, \beta, n}^{\mu} &=
\delta_{0, \beta, n} \ast \delta_{1, 0, 0} \\
&= \sum_{\begin{subarray}{c}
(\beta_1, n_1) +(\beta_2, n_2)=(\beta, n), \\
n_i=\mu(\omega \cdot \beta_i)
\end{subarray}}
\delta_{1, \beta_1, n_1}^{\alpha} \ast \delta_{0, \beta_2, n_2}. 
\end{align*}
\end{prop}
\begin{proof}
The result obviously follows from Corollary~\ref{cor:HN}
and the arguments given in~\cite[Theorem~5.11]{Joy4}. 
\end{proof}

\subsection{Lie algebra homomorphism}
Following~\cite{JS}, we can 
construct a Lie algebra homomorphism from 
a certain Lie subalgebra of 
$H(\mu, d)$ to a Lie algebra defined 
by the Euler pairing on $\aA(\mu, d)$. 
Applying the Lie algebra homomorphism to 
the formula in Proposition~\ref{prop:ident}, 
we obtain a formula relating invariants 
counting parabolic stable pairs to 
generalized DT invariants introduced in~\cite{JS}, \cite{K-S}. 

First by~\cite[Definition~5.13]{Joy2}, 
there is a Lie subalgebra, 
\begin{align*}
H^{\mathrm{Lie}}(\mu, d) \subset H(\mu, d), 
\end{align*}
consisting of elements 
$[\rho \colon \mathrsfs{X} \to \mathrsfs{A}(\mu, d)]$, 
supported on `virtual indecomposable objects'. 
The definition of the virtual indecomposable objects 
is very complicated and we omit its detail. The only 
property we need is that, 
\begin{align*}
\epsilon_{r, \beta, n}^{\alpha} \in H^{\mathrm{Lie}}(\mu, d),
\end{align*} 
for any element $\epsilon_{r, \beta, n}^{\alpha}$
constructed in the previous subsection. 

Next we define the Lie algebra $C(\mu, d)$. 
Let $\Gamma(\mu)$ be the abelian group 
given by (\ref{Gmu}). 
We have the map, 
\begin{align*}
\cl \colon \aA(\mu, d) \to \Gamma(\mu), 
\end{align*}
defined by 
\begin{align*}
\cl(N_{H/X}^{\oplus r}[-1] \to F) 
=(r, [F], \chi(F)). 
\end{align*}
Let $\chi \colon \Gamma(\mu) \times \Gamma(\mu) \to \mathbb{Z}$
be a bilinear anti-symmetric pairing given by 
\begin{align*}
\chi((r_1, \beta_1, n_1), (r_2, \beta_2, n_2))=
r_2(\beta_1 \cdot H) - r_1(\beta_2 \cdot H). 
\end{align*}
We have the following lemma. 
\begin{lem}
For $E_i \in \aA(\mu, d)$ with $i=1, 2$ satisfying 
\begin{align*}
\cl(E_i)=(r_i, \beta_i, n_i), \quad 
\omega \cdot (\beta_1 +\beta_2) \le d,
\end{align*} we have 
\begin{align}\notag
\chi(E_1, E_2) &=\dim \Hom(E_1, E_2)- \dim \Ext^1(E_1, E_2) \\
\label{chi1}
& \qquad -\dim \Hom(E_2, E_1) +\dim \Ext^1(E_2, E_1). 
\end{align}
\end{lem}
\begin{proof}
By Lemma~\ref{lem:exseq}, we have 
\begin{align}\notag
&\dim \Hom(E_2, E_1)-\dim \Ext^1(E_2, E_1)  \\
\label{chi2}
&=\dim \Hom(F_2, F_1) -\dim \Ext^1(F_2, F_1)
+r_1 r_2 - r_1 (\beta_2 \cdot H). 
\end{align}
On the other hand, by the Serre
duality and Riemann-Roch 
theorem on $X$, we have 
\begin{align} \notag
&\dim \Hom(F_1, F_2)- \dim \Ext^1(F_1, F_2) \\
\label{chi3}
& \qquad + \dim \Ext^1(F_2, F_1) -\dim \Hom(F_2, F_1)=0. 
\end{align}
The formula (\ref{chi1}) 
follows from (\ref{chi2}) and (\ref{chi3}). 
\end{proof}
Let us consider the following 
$\mathbb{Q}$-vector space, 
\begin{align}\label{C(mu)}
C(\mu) \cneq \bigoplus_{(r, \beta, n) \in \Gamma(\mu)}
\mathbb{Q} c_{(r, \beta, n)}. 
\end{align}
There is the Lie algebra 
structure on the
 $\mathbb{Q}$-vector space $C(\mu)$ 
by 
\begin{align*}
[c_{v_1}, c_{v_2}] =(-1)^{\chi(v_1, v_2)} \chi(v_1, v_2)
c_{v_1 +v_2},
\end{align*}
for $v_1, v_2 \in \Gamma(\mu)$. Let $I_d \subset C(\mu)$
be the ideal of $C(\mu)$, 
generated by $c_{(r, \beta, n)}$ with $\omega \cdot \beta >d$. 
The Lie algebra $C(\mu, d)$ is defined to be the 
quotient, 
\begin{align*}
C(\mu, d) \cneq C(\mu)/I_{d}. 
\end{align*}
Finally we recall that the Behrend function~\cite{Beh}
on $\mathbb{C}$-schemes are naturally generalized to 
those on Artin stacks over $\mathbb{C}$. (cf.~\cite[Theorem~5.12]{JS}.)
Let 
\begin{align*}
\nu_{\mathrsfs{A}} \colon \mathrsfs{A}(\mu, d) \to \mathbb{Z},
\end{align*}
be the Behrend function on $\mathrsfs{A}(\mu, d)$. 
The argument of Joyce-Song~\cite{JS} 
is applied in our situation, and the
following result holds. 
\begin{thm}\label{thm:Lie}
There is a Lie algebra homomorphism, 
\begin{align*}
\Upsilon \colon H^{\mathrm{Lie}}(\mu, d) 
\to C(\mu, d), 
\end{align*}
such that for an element
\begin{align*}
u=
\left[ 
[M/\mathbb{C}^{\ast}] \stackrel{\rho}{\to} \mathrsfs{A}_{r, \beta, n}
\right] \in H^{\mathrm{Lie}}(\mu, d), 
\end{align*}
where $M$ is a $\mathbb{C}$-scheme 
with a trivial $\mathbb{C}^{\ast}$-action, 
we have 
\begin{align}\label{form:Up}
\Upsilon(u)=\left( \int_{M} \rho^{\ast} \nu_{\mathrsfs{A}} d\chi   \right)
c_{(r, \beta, n)}. 
\end{align}
\end{thm}
\begin{proof}
Although $\mathrsfs{A}(\mu, d)$ is not an 
abelian category, the same proof of~\cite[Theorem~5.12]{JS} works. 
We only have to notice the following two things. 

First for $v_i=(r_i, \beta_i, n_i) \in \Gamma(\mu)$
with $i=1, 2$ and $\omega \cdot (\beta_1 +\beta_2) >d$, we have 
\begin{align*}
[u_1, u_2]=0, \quad [c_{v_1}, c_{v_2}]=0, 
\end{align*}
for $u_i \in H_{r_i, \beta_i, n_i}$. 
Therefore we only need to consider the Lie brackets
in the case of $\omega \cdot (\beta_1 +\beta_2) \le d$. 
In that case, we can discuss as if $\aA(\mu, d)$
were an abelian category, and use the argument in~\cite[Theorem~5.12]{JS}. 

The next thing is to show the version of~\cite[Theorem~5.3]{JS}
in our situation. Namely we need to check that 
the moduli stack $\mathrsfs{A}(\mu, d)$
is analytically locally written as a critical 
locus of some holomorphic function, up
to some group action. 

It is enough to check this for the substack 
$\mathrsfs{A}_{r, \beta, n} \subset \mathrsfs{A}(\mu, d)$.
 Let $\mathrsfs{M}_{\beta, n}$ be 
the moduli stack of one dimensional 
semistable sheaves $F$ on $X$ with $[F]=\beta$, $\chi(F)=n$. 
We have the forgetting 1-morphism, 
\begin{align*}
{\bf f} \colon 
\mathrsfs{A}_{r, \beta, n} \to \mathrsfs{M}_{\beta, n},
\end{align*}
sending $(N_{H/X}^{\oplus r}[-1] \to F)$ to $F$. 
On the other hand, 
in the notation of Subsection~\ref{subsec:const}, 
Subsection~\ref{subsec:stackA}, 
the stack $\mathrsfs{M}_{\beta, n}$
is written as the quotient stack, 
\begin{align*}
\mathrsfs{M}_{\beta, n} \cong [U/\GL(V)].
\end{align*}
Combined with the description (\ref{Astack})
of $\mathrsfs{A}_{r, \beta, n}$, 
the 1-morphism ${\bf f}$ is 
induced from the natural 
projection $R_{U}^{(r)} \to U$, 
\begin{align*}
{\bf f} \colon 
\left[R_U^{(r)}/\left(\GL(V) \times \GL(r, \mathbb{C}) \right)   \right]
\to [U/\GL(V)]. 
\end{align*}
The morphism $R_U^{(r)} \to U$
is a Zariski-locally 
trivial fibration with fiber $(\mathbb{C}^{\beta \cdot H})^{\times r}$. 
Thus the morphism 
${\bf f}$ is also a Zariski locally trivial 
fibration with fiber the 
quotient stack 
\begin{align*}
[(\mathbb{C}^{\beta \cdot H})^{\times r}/\GL(r, \mathbb{C})].
\end{align*} 
By~\cite[Theorem~5.3]{JS}, the stack 
$\mathrsfs{M}_{\beta, n}$ is analytically 
locally written as a critical locus of some holomorphic
function up to some group action.
Hence the same 
property holds for the stack $\mathrsfs{A}_{r, \beta, n}$
by the above argument. 
\end{proof}

\subsection{Product expansion formula}
Let $\Upsilon$ be the Lie algebra homomorphism 
discussed in Theorem~\ref{thm:Lie}.
If $\alpha \in \mathbb{Q}$ satisfies 
$\alpha>\mu$, then we have 
$\delta^{\alpha}_{1, \beta, n}=\epsilon^{\alpha}_{1, \beta, n}$
and
\begin{align}\label{Up:D}
\Upsilon(\epsilon^{\alpha}_{1, \beta, n})
=-\mathrm{DT}_{n, \beta}^{\rm{par}} c_{(1, \beta, n)}, 
\end{align} 
by Definition~\ref{defi:parainv},
the isomorphism (\ref{one:isom}) and the formula (\ref{form:Up}). 
Here we need to change the sign 
since the Behrend function on a scheme $M$
and that on $[M/\mathbb{C}^{\ast}]$ differ 
by a sign. 

By applying 
$\Upsilon$ to the element 
$\epsilon_{0, \beta, n} \in H^{\rm{Lie}}(\mu, d)$,  
we obtain the invariant $N_{n, \beta} \in \mathbb{Q}$. 
\begin{defi}\label{defi:Nn}
For $(0, \beta, n) \in \Gamma(\mu)$ with
$\omega \cdot \beta \le d$, 
the invariant $N_{n, \beta} \in \mathbb{Q}$ 
is defined by   
\begin{align}\label{via:N}
\Upsilon(\epsilon_{0, \beta, n})=-N_{n, \beta} c_{(0, \beta, n)}. 
\end{align}
\end{defi}
The invariant $N_{n, \beta} \in \mathbb{Q}$ is 
nothing but a generalized DT invariant introduced in~\cite{JS}, \cite{K-S}, 
counting $\omega$-semistable one dimensional 
sheaves $F$ satisfying $[F]=\beta$, $\chi(F)=n$. 
In fact, by Remark~\ref{donot:depend} (ii), 
the stack defining $\delta_{0, \beta, n}$ 
is the moduli stack of the above $\omega$-semistable 
sheaves, and it is obvious that the invariant
$N_{n, \beta}$ defined via (\ref{via:N}) 
coincides with the invariant 
given in~\cite{JS}. 
Also see~\cite[Section~4]{Tsurvey} for the explanation of 
the construction of $N_{n, \beta} \in \mathbb{Q}$. 
\begin{rmk}
The restriction $\omega \cdot \beta \le d$ is 
not necessary in defining $N_{n, \beta}$
in~\cite{JS}. The restriction $\omega \cdot \beta \le d$
is put in Definition~\ref{defi:Nn} since 
$\epsilon_{0, \beta, n}=0$ if $\omega \cdot \beta>0$ 
so $N_{n, \beta}$ cannot be defined via (\ref{via:N}) 
without the above restriction. 
\end{rmk}
\begin{rmk}\label{rmk:depend}
As explained in~\cite[Theorem~6.16]{JS}, the invariant 
$N_{n, \beta}$ does not depend on a choice of $\omega$. 
Of course it also does not depend on $H$, as 
we do not use $H$ to define $\omega$-semistable sheaves.
\end{rmk}
\begin{rmk}\label{rmk:coprime}
Suppose that $n$ and $\beta$ are coprime. 
Then there is a $\mathbb{Q}$-ample divisor $\omega$
such that the moduli space of $\omega$-stable 
one dimensional sheaves $F$ with $[F]=\beta$, 
$\chi(F)=n$ is a projective scheme. 
(i.e. there is no strictly $\omega$-semistable 
sheaves.)
If $M_n(X, \beta)$ is such a moduli space, then $N_{n, \beta}$
is given by 
\begin{align*}
N_{n, \beta}=\int_{M_n(X, \beta)} \nu_{M} d\chi, 
\end{align*}
where $\nu_M$ is the Behrend function on $M_n(X, \beta)$. 
\end{rmk}
Recall that we defined the series 
$\mathrm{DT}^{\rm{par}}(\mu, d)$
as an element in $\Lambda_{\le d}$ in Subsection~\ref{subsec:Count}. 
Applying the Lie algebra homomorphism $\Upsilon$ in 
Theorem~\ref{thm:Lie}, we obtain a formula 
relating $\mathrm{DT}^{\rm{par}}_{n, \beta}$
and $N_{n, \beta}$. 
The following result follows from the 
same arguments as in~\cite[Theorem~5.24]{JS}, \cite[Theorem~4.7]{Tolim2}, 
\cite[Theorem~5.8]{Tcurve1}. 
For the reader's convenience, we provide the argument. 
\begin{thm}\label{thm:DTpar}
We have the following formula in $\Lambda_{\le d}$, 
\begin{align}\label{form:thm2}
\mathrm{DT}^{\rm{par}}(\mu, d)
=\prod_{\begin{subarray}{c}
\beta>0, \\
 n/\omega \cdot \beta=\mu
\end{subarray}}
\exp\left( (-1)^{\beta \cdot H -1} 
N_{n, \beta}q^n t^{\beta}  \right)^{\beta \cdot H}. 
\end{align}
Here $\beta>0$ means that $\beta$ is 
a homology class of an effective one cycle on $X$.
\end{thm}
\begin{proof}
Let us take $\alpha \in \mathbb{Q}$ satisfying 
$\alpha>\mu$. 
We set $\delta^{\alpha} \in H^{\rm{Lie}}(\mu, d)$ to be 
\begin{align*}
\delta^{\alpha} \cneq
\sum_{n=\mu(\omega \cdot \beta)}\delta_{1, \beta, n}^{\alpha}
=\sum_{n=\mu(\omega \cdot \beta)}\epsilon_{1, \beta, n}^{\alpha}. 
\end{align*}
Note that $\delta_{1, \beta, n}^{\alpha}=\epsilon_{1, \beta, n}^{\alpha}=
\delta_{0, \beta, n}=\epsilon_{0, \beta, n}=0$
if $\omega \cdot \beta>d$. 
By Proposition~\ref{prop:ident}, we have the 
identity in $H(\mu, d)$, 
\begin{align}\label{delta:id}
\left( \sum_{n=\mu(\omega \cdot \beta)} \delta_{0, \beta, n} \right) \ast \delta_{1, 0, 0}
=\delta^{\alpha} \ast \left( \sum_{n=\mu(\omega \cdot \beta)} \delta_{0, \beta, n} \right). 
\end{align}
We set $\mathfrak{E} \in H^{\rm{Lie}}(\mu, d)$ to be 
\begin{align*}
\mathfrak{E}=\sum_{n=\mu(\omega \cdot \beta)}
\epsilon_{0, \beta, n}. 
\end{align*}
Then by (\ref{Or:equiv}) and (\ref{delta:id}), we have 
\begin{align}\notag
\delta^{\alpha} &= \exp(\mathfrak{E}) \ast \delta_{1, 0, 0}
 \ast \exp(\mathfrak{E})^{-1} \\
\label{Ad}
&=\sum_{k \ge 0} 
\frac{1}{k!} (\mathrm{Ad}_{\mathfrak{E}})^{k}(\delta_{1, 0, 0}). 
\end{align} 
Here we have used the Baker-Campbell-Hausdorff 
formula in (\ref{Ad}). 
Applying $\Upsilon$ to (\ref{Ad}),
using (\ref{Up:D}), (\ref{via:N})
and setting $\mathrm{DT}_{0, 0}^{\rm{par}}=1$,   
we have the following 
identities in $C(\mu, d)$, 
\begin{align}\notag
&\sum_{n=\mu(\omega \cdot \beta)}\mathrm{DT}_{n, \beta}^{\rm{par}}
c_{(1, \beta, n)} \\
\notag
&=\sum_{k\ge 0} \frac{(-1)^k}{k!}
\sum_{\begin{subarray}{c}
\beta_1, \cdots, \beta_k \in H_2(X, \mathbb{Z}), \\
n_1, \cdots, n_k \in \mathbb{Z}, \\
n_i=\mu(\omega \cdot \beta_i)
\end{subarray}}
\prod_{i=1}^{k} N_{n_i, \beta_i}
\cdot 
\mathrm{Ad}_{c_{(0, \beta_1, n_1)}} \circ
 \cdots \circ
\mathrm{Ad}_{c_{(0, \beta_k, n_k)}}(c_{(1, 0, 0)}) \\
\label{ident1}
&= \sum_{k\ge 0} \frac{1}{k!}
\sum_{\begin{subarray}{c}
\beta_1, \cdots, \beta_k \in H_2(X, \mathbb{Z}), \\
n_1, \cdots, n_k \in \mathbb{Z}, \\
n_i=\mu(\omega \cdot \beta_i)
\end{subarray}}
\prod_{i=1}^{k} (-1)^{\beta_i \cdot H -1}(\beta_i \cdot H) N_{n_i, \beta_i}
\cdot 
c_{(1, \sum_{i=1}^k \beta_i, \sum_{i=1}^{k} n_i)}.  
\end{align}
The formula (\ref{form:thm2}) obviously follows from (\ref{ident1}). 
\end{proof}

\section{Multiple cover formula}\label{sec:mult}
In this section, we discuss relationship 
between Theorem~\ref{thm:DTpar} and the 
conjectural multiple cover formula 
of the generalized DT invariants $N_{n, \beta}$. 
\subsection{Conjectural multiple cover formula}\label{subsec:conjectural}
In this subsection, we recall a conjectural 
multiple cover formula of the 
invariants $N_{n, \beta}$. 
The statement of the conjecture is as follows. 
\begin{conj}{\bf\cite[Conjecture~6.20]{JS}, 
\cite[Conjecture~6.3]{Tsurvey}}\label{conj:mult2}
We have the following formula, 
\begin{align}\label{form:mult:N}
N_{n, \beta}=\sum_{k\ge 1, k|(n, \beta)}
\frac{1}{k^2}N_{1, \beta/k}. 
\end{align}
\end{conj}
The conjecture is motivated by the 
strong rationality conjecture of 
the generating series of PT invariants~\cite{PT}. 
As we recalled in Subsection~\ref{subsec:relation}, a PT stable 
pair consists of data, 
\begin{align}\label{PT:mult}
(F, s), \quad s \colon \oO_X \to F, 
\end{align}
where $F$ is a pure one dimensional sheaf 
and $s$ is surjective in dimension one. 
For $\beta \in H_2(X, \mathbb{Z})$ and $n\in \mathbb{Z}$, 
The moduli space of PT stable pairs (\ref{PT:mult}) 
with $[F]=\beta$ and $\chi(F)=n$ is denoted by 
$P_n(X, \beta)$. 
The PT invariant is defined by 
\begin{align}\notag
P_{n, \beta} \cneq \int_{P_n(X, \beta)}
\nu_{P} d\chi \in \mathbb{Z}. 
\end{align}
Here $\nu_{P}$ is the Behrend function on 
$P_n(X, \beta)$. 

Let $\mathrm{PT}_{\beta}(X)$ and $\mathrm{PT}(X)$ 
be the generating series, 
\begin{align}\notag
\mathrm{PT}_{\beta}(X) &\cneq 
\sum_{n \in \mathbb{Z}}P_{n, \beta}q^n, \\
\mathrm{PT}(X) & \cneq 
\sum_{\beta \in H_2(X, \mathbb{Z})} \mathrm{PT}_{\beta}(X)t^{\beta}. 
\end{align}
The main conjecture by Pandharipande-Thomas~\cite{PT}
is an equivalence between 
the generating series of 
Gromov-Witten 
invariants and the generating 
series of PT invariants $\mathrm{PT}(X)$
after a suitable variable change. 
In order to make the variable change 
well-defined, the series $\mathrm{PT}_{\beta}(X)$
should satisfy a (weak) rationality property: 
i.e. $\mathrm{PT}_{\beta}(X)$ should be 
the Laurent expansion of a rational function of 
$q$, invariant under $q \leftrightarrow 1/q$. 

On the other hand, if we believe
GW/PT correspondence, 
then the series $\mathrm{PT}(X)$ is 
expected to be written as a 
certain infinite product expansion, 
called \textit{Gopakumar-Vafa form}. 
The conjecture is formulated in the following way. 
(cf.~\cite[Equation~(18)]{Katz2}, \cite[Conjecture~6.2]{Tsurvey}.)
\begin{conj}\label{conj:PTGV}
There are integers 
\begin{align*}
n_{g}^{\beta} \in \mathbb{Z}, \mbox{ for }
g \ge 0, \ \beta \in H_2(X, \mathbb{Z}), 
\end{align*}
such that we have 
\begin{align}\label{GVform}
\PT(X)=\prod_{\beta >0}\prod_{j=1}^{\infty}
(1-(-q)^j t^{\beta})^{j n_{0}^{\beta}}
 \cdot \prod_{g=1}^{\infty}\prod_{k=0}^{2g-2}
(1-(-q)^{g-1-k}t^{\beta})^{(-1)^{k+g}n_{g}^{\beta}
\left(\begin{subarray}{c}
2g-2 \\
k
\end{subarray} \right)}. 
\end{align}
\end{conj}
The above conjecture also implies the 
weak rationality of $\mathrm{PT}_{\beta}(X)$, and 
in fact it is nothing but Pandharipande-Thomas's
strong rationality conjecture~\cite[Conjecture~3.14]{PT}.
 
Now in~\cite{Tolim2}, \cite{BrH}, 
the weak rationality of $\mathrm{PT}_{\beta}(X)$
is solved. The idea is to relate the
invariants $P_{n, \beta}$ with 
 $N_{n, \beta}$
and other invariants $L_{n, \beta}$, 
and use some properties of the latter invariants. 
More precisely we have the following theorem. 
\begin{thm}{\bf\cite{Tolim2} (Euler characteristic version), \cite{BrH}}
There are invariants 
$L_{n, \beta} \in \mathbb{Q}$ satisfying 
\begin{align*}
L_{n, \beta}=L_{-n, \beta}, \quad L_{n, \beta}=0,
\mbox{ for } \lvert n \rvert \gg 0,
\end{align*}
 such that 
the following formula holds:
\begin{align}\label{form:rat}
\mathrm{PT}(X)=
\prod_{n>0, \beta>0}
\exp \left( (-1)^{n-1} N_{n, \beta}q^n t^{\beta}\right)^{n} \left( \sum_{n, \beta}L_{n, \beta}q^n t^{\beta} \right).  
\end{align}
\end{thm}
The above theorem implies 
the weak rationality of $\mathrm{PT}_{\beta}(X)$. 
(cf.~\cite[Corollary~4.8]{Tolim2}.)
The formula (\ref{form:rat}) 
is weaker than the formula (\ref{GVform}),
but it reduces Conjecture~\ref{conj:PTGV} to Conjecture~\ref{conj:mult2}. 
Namely we have the following corollary.  
\begin{cor}{\bf \cite[Theorem~1.1]{BrH}, \cite[Theorem~6.4]{Tsurvey}}

(i) The series $\mathrm{PT}_{\beta}(X)$ is the Laurent 
expansion of a rational function of $q$, invariant 
under $q \leftrightarrow 1/q$. 

(ii) Conjecture~\ref{conj:PTGV} holds if and only if
 Conjecture~\ref{conj:mult2} holds. In this case, we have 
\begin{align*}
n_{0, \beta}=N_{1, \beta}. 
\end{align*}
\end{cor}

\subsection{Multiple cover formula via parabolic stable pairs}
Recall that we have given a formula relating 
invariants counting parabolic stable pairs to 
the invariants $N_{n, \beta}$. In this subsection, 
we see that conjecture~\ref{conj:mult2}
is also equivalent to a conjectural product expansion 
formula of the series $\mathrm{DT}^{\rm{par}}(\mu, d)$. 
Namely we have the following proposition. 
\begin{prop}\label{prop:mult:para}
The formula (\ref{form:mult:N}) holds for any $(n, \beta)$
with $\beta \cdot \omega \le d$
and $n/\beta \cdot \omega =\mu$ if and only if 
the following formula holds in $\Lambda_{\le d}$, 
\begin{align}\label{cor:form2}
\mathrm{DT}^{\rm{par}}(\mu, d)
=\prod_{\begin{subarray}{c}
\beta>0, \\
 n/\omega \cdot \beta=\mu
\end{subarray}}
\left(1-(-1)^{\beta \cdot H} q^n t^{\beta}\right)^{(\beta \cdot H)N_{1, \beta}}
\end{align}
\end{prop}
\begin{proof}
By taking the logarithm of both sides of 
(\ref{form:thm2}), we have 
\begin{align}\label{logPar}
\log \mathrm{DT}^{\rm{par}}(\mu, d)=
\sum_{\begin{subarray}{c}
\beta>0, \\
n/\omega \cdot \beta =\mu
\end{subarray}}
(-1)^{\beta \cdot H -1}(\beta \cdot H)
N_{n, \beta}q^n t^{\beta}. 
\end{align}
On the other hand, the logarithm 
of the RHS of (\ref{cor:form2}) is 
\begin{align}\notag
&\sum_{\begin{subarray}{c}
\beta>0, \\
n/\omega \cdot \beta =\mu
\end{subarray}}
(\beta \cdot H)
N_{1, \beta} \log \left( 1-(-1)^{\beta \cdot H}q^n t^{\beta} \right) \\
&=\label{logPar2}
\sum_{\begin{subarray}{c}
\beta>0, \\
n/\omega \cdot \beta =\mu
\end{subarray}}
\sum_{k\ge 1, k|(n, \beta)}
\frac{(-1)^{\beta \cdot H -1}}{k^2}
(\beta \cdot H) N_{1, \beta/k}q^n t^{\beta}. 
\end{align}
Comparing (\ref{logPar}) with (\ref{logPar2}), we obtain the result. 
\end{proof}
\begin{rmk}\label{enough:check}
The proof of Proposition~\ref{prop:mult:para} shows that, 
in order to check the formula (\ref{form:mult:N}) 
for a specific $(\beta, n)$, 
it is enough to check the equality of the
coefficients of $q^{n}t^{\beta}$
of $\log \mathrm{DT}^{\rm{par}}(\mu, d)$
and the logarithm of the RHS of (\ref{cor:form2}). 
\end{rmk}
By the above proposition, the following 
conjecture is equivalent to both of Conjecture~\ref{conj:mult2} 
and Conjecture~\ref{conj:PTGV}. 
\begin{conj}\label{conj:prod:exp}
We have the following formula in $\Lambda_{\le d}$, 
\begin{align}\notag
\mathrm{DT}^{\rm{par}}(\mu, d)
=\prod_{\begin{subarray}{c}
\beta>0, \\
 n/\omega \cdot \beta=\mu
\end{subarray}}
\left(1-(-1)^{\beta \cdot H} 
q^n t^{\beta}\right)^{(\beta \cdot H)N_{1, \beta}}.
\end{align}
\end{conj}
By Remark~\ref{enough:check}, the above 
conjecture is equivalent to the following: 
if we set $\widehat{\DT}_{n, \beta}^{\rm{par}} \in \mathbb{Q}$
by 
\begin{align}\label{DTtilde}
\log \DT^{\rm{par}}(\mu, d)=
\sum_{n, \beta} \widehat{\DT}_{n, \beta}^{\rm{par}}q^n t^{\beta}, 
\end{align}
then we have the formula, 
\begin{align*}
\widehat{\DT}_{n, \beta}^{\rm{par}}=
(-1)^{\beta \cdot H -1}(\beta \cdot H) \sum_{k\ge 1, k|(n, \beta)}
\frac{1}{k^2} N_{1, \beta/k}. 
\end{align*}

\subsection{Local parabolic stable pair invariants}
In the following subsections, we study the local 
version of parabolic stable pairs and 
relevant results. All the 
arguments are similar to the global 
case, and we omit several details. 

In what follows, we fix a reduced subscheme,
\begin{align*}
i\colon C \subset X, 
\end{align*}
with $\dim C=1$. 
We also fix a divisor in $X$, 
\begin{align}\label{choice:H}
H \in \lvert \oO_X(h) \rvert, \quad h>0, 
\end{align}
which intersects with $C$ transversally. 
Of course, any one cycle on $X$ supported on 
$C$ intersects with $H$ transversally. 
On the other hand,
 we do not assume the transversality 
for the intersection of $H$ with curves
of bounded degree other than $C$. 
In this sense, the way we choose for (\ref{choice:H}) is different 
from the way for $H \subset X$ in Lemma~\ref{lem:trans}. 
The former one depends on the curve $C$, while 
the latter one depends on the degree $d \in \mathbb{Z}_{>0}$. 

Let $M_n^{\rm{par}}(X, \beta)$ be the moduli space 
of parabolic stable pairs
with respect to the above choice of $H$. 
Since $H$ may not satisfy the condition in Lemma~\ref{lem:trans}, 
the moduli space $M_n^{\rm{par}}(X, \beta)$
may not be projective, 
but it 
is at least a quasi-projective 
scheme as we mentioned in Remark~\ref{rmk:quasi}. 
Let $\Chow_{\beta}(X)$
be the Chow variety parameterizing 
one cycles $C' \subset X$ with $[C']=\beta$. 
There is a Hilbert-Chow type morphism, 
\begin{align*}
\pi_{\beta} \colon 
M_n^{\rm{par}}(X, \beta) \to \Chow_{\beta}(X), 
\end{align*}
sending a parabolic stable pair $(F, s)$
to the one cycle $[F]$. 

Let $C_1, \cdots, C_N$ be the set of 
irreducible components of $C$. 
We have 
\begin{align*}
H_2(C, \mathbb{Z}) \cong \bigoplus_{i=1}^{N}
\mathbb{Z}[C_i], 
\end{align*}
and we can identify $H_2(C, \mathbb{Z})$
with the group of one cycles on $X$ supported on $C$.
The effective cone is defined by, 
\begin{align*}
H_2(C, \mathbb{Z})_{>0}
\cneq \left\{ \sum_{i=1}^{N}a_i [C_i] :
a_i \ge 0 \right\} \setminus \{0\} 
\subset H_2(C, \mathbb{Z}). 
\end{align*}
For each $\gamma \in H_2(C, \mathbb{Z})_{>0}$, 
we can regard it as 
\begin{align*}
\gamma \in \Chow_{i_{\ast}\gamma}(X). 
\end{align*} 
We define the local parabolic stable pair 
invariant in the following way. 
\begin{defi}\label{loc:inv}
For each $\gamma \in H_2(C, \mathbb{Z})_{>0}$, we define 
$\mathrm{DT}_{n, \gamma}^{\rm{par}} \in \mathbb{Z}$
to be 
\begin{align*}
\mathrm{DT}_{n, \gamma}^{\rm{par}}
\cneq \int_{\pi_{i_{\ast}\gamma}^{-1}(\gamma)} \nu_{M} d\chi. 
\end{align*}
\end{defi} 
Here we note that $\nu_M$ is a Behrend function on 
$M_n^{\rm{par}}(X, i_{\ast}\gamma)$, not on 
the fiber $\pi_{i_{\ast}\gamma}^{-1}(\gamma)$. 

\subsection{Local generalized DT invariants}
We can also define the local version of
generalized DT invariants
in a way similar to Definition~\ref{defi:Nn}. 
Namely we replace the category 
$\aA(\mu, d)$ by the category of pairs, 
\begin{align}\label{pair:NF}
N_{H/X}^{\oplus r}[-1] \to F, 
\end{align}
where $F$ is an $\omega$-semistable sheaf supported on $C$, 
satisfying $\mu_{\omega}(F)=\mu$. 
The category consisting of the above 
pairs is denoted by $\aA(\mu, C)$. 
The moduli stack of objects in $\aA(\mu, C)$
can be constructed as in Subsection~\ref{subsec:stackA}. 
Namely let $\mathrsfs{A}(\mu)$ be the stack of 
pairs $N_{H/X}^{\oplus r}[-1] \to F$, where 
$F$ is $\mu_{\omega}$-semistable 
with $\mu_{\omega}(F)=\mu$, but 
not necessary supported on $C$. 
By the arguments in Subsection~\ref{subsec:stackA} and 
Remark~\ref{rmk:quasi}, the 
stack $\mathrsfs{A}(\mu)$ 
can be shown to be an Artin stack 
locally of finite type over $\mathbb{C}$. 
The desired stack of objects in $\aA(\mu, C)$
is the closed substack, 
\begin{align*}
\mathrsfs{A}(\mu, C) \subset \mathrsfs{A}(\mu).
\end{align*}
The stack $\mathrsfs{A}(\mu, C)$ 
decomposes as 
\begin{align*}
\mathrsfs{A}(\mu, C)=\coprod_{(r, \gamma, n) \in 
\Gamma(\mu, C)}
\mathrsfs{A}_{r, \gamma, n},
\end{align*}
where $\mathrsfs{A}_{r, \gamma, n}$ is 
the stack of pairs (\ref{pair:NF}) 
with $[F]=\gamma$ as a one cycle supported on $C$,
and $\Gamma(\mu, C)$
is defined similarly to (\ref{Gmu}), 
by replacing $H_2(X, \mathbb{Z})$ by $H_2(C, \mathbb{Z})$. 
 
Hence we have the Hall type algebra $H(\mu, C)$
and the Lie subalgebra of virtual indecomposable objects, 
\begin{align*}
H^{\rm{Lie}}(\mu, C) \subset H(\mu, C).
\end{align*} 
The Lie algebra $C(\mu, C)$ is also 
defined similarly to (\ref{C(mu)}), just by 
replacing $H_2(X, \mathbb{Z})$ by $H_2(C, \mathbb{Z})$,
\begin{align*}
C(\mu, C)=\bigoplus_{(r, \gamma, n) \in
\Gamma(\mu, C)}
\mathbb{Q}c_{(r, \gamma, n)}. 
\end{align*}
Here we do not take a quotient as in defining $C(\mu, d)$. 

The Lie algebra homomorphism 
\begin{align*}
\Upsilon \colon H^{\rm{Lie}}(\mu, C)
\to C(\mu, C),
\end{align*}
can be similarly constructed as in 
Theorem~\ref{thm:Lie}. 
The only point we have to notice is that
we use the Behrend function on $\mathrsfs{A}(\mu)$, 
not on $\mathrsfs{A}(\mu, C)$. 
As in the proof of Theorem~\ref{thm:Lie}, 
the stack $\mathrsfs{A}(\mu)$ is analytic locally 
written as a critical locus of some holomorphic
function, hence the same argument 
in Theorem~\ref{thm:Lie} can be applied. 
Also for an element
\begin{align*}
u=\left[ [M/\mathbb{C}^{\ast}] \stackrel{\rho}{\to}
\mathrsfs{A}_{r, \gamma, n} \right] \in H^{\rm{Lie}}(\mu, C), 
\end{align*}
where $M$ is a $\mathbb{C}$-scheme with a trivial 
$\mathbb{C}^{\ast}$-action, we have 
\begin{align}\notag
\Upsilon(u)=\left( \int_{M} \rho^{\ast} \nu_{\mathrsfs{A}} d\chi   \right)
c_{(r, \gamma, n)}. 
\end{align}
Here $\nu_{\mathrsfs{A}}$ is the 
Behrend function on $\mathrsfs{A}(\mu)$
restricted to $\mathrsfs{A}_{r, \gamma, n}$, 
which may be different from that on $\mathrsfs{A}_{r, \gamma, n}$. 

The elements, 
\begin{align*}
\delta_{r, \gamma, n}^{\alpha} \in H(\mu, C), \ 
\epsilon_{r, \gamma, n}^{\alpha} \in H^{\rm{Lie}}(\mu, C),
\end{align*}
can be defined similarly to (\ref{delta}), (\ref{epsilon}), 
using similar $\widehat{\mu}_{\alpha}$-semistability on $\aA(\mu, C)$
and the moduli stack of $\widehat{\mu}_{\alpha}$-semistable 
pairs (\ref{pair:NF}). 
The local generalized DT invariant is defined 
as follows. 
\begin{defi}
For $(0, \gamma, n) \in \Gamma(\mu, C)$, the invariant 
$N_{n, \gamma} \in \mathbb{Q}$ is defined 
by 
\begin{align*}
\Upsilon(\epsilon_{0, \gamma, n}) =-N_{n, \gamma} c_{(0, \gamma, n)}. 
\end{align*}
\end{defi} 
\begin{rmk}
Similarly to Remark~\ref{rmk:depend}, 
the invariant $N_{n, \gamma}$ does not depend on $\omega$. 
\end{rmk}
\begin{rmk}
Suppose that $n$ and $\gamma$ are coprime. 
Then there is an ample $\mathbb{Q}$-divisor $\omega$
such that the moduli space of $\omega$-stable 
sheaves $F$ with $\chi(F)=n$ and $[F]=\gamma$
as a one cycle is a closed subscheme, 
\begin{align*}
M_n(C, \gamma) \subset M_n(X, i_{\ast}\gamma). 
\end{align*}
As in Remark~\ref{rmk:coprime}, the invariant 
$N_{n, \gamma}$ is given by 
\begin{align}\label{N_n}
N_{n, \gamma} =\int_{M_n(C, \gamma)} \nu_{M} d\chi, 
\end{align}
where $\nu_M$ is the Behrend function on 
$M_n(X, i_{\ast}\gamma)$, not on $M_n(C, \gamma)$. 
\end{rmk}

\subsection{Generating series of local invariants}
Let $\mathrm{DT}^{\rm{par}}(\mu, C)$ be the 
generating series, 
\begin{align*}
\mathrm{DT}^{\rm{par}}(\mu, C)
\cneq 1+ \sum_{\begin{subarray}{c}
n\in \mathbb{Z}, \ 
\gamma \in H_2(C, \mathbb{Z})_{>0}, \\
n/\omega \cdot \gamma=\mu
\end{subarray}}
\mathrm{DT}_{n, \gamma}^{\rm{par}}
q^n t^{\gamma} \in \Lambda_{C}.
\end{align*}
Here $\Lambda_{C}$ is defined by 
\begin{align*}
\Lambda_{C} \cneq \prod_{
n \in \mathbb{Z}, 
\gamma \in H_2(C, \mathbb{Z})_{>0}}
\mathbb{Q} q^n t^{\gamma}. 
\end{align*}
As an analogy of Theorem~\ref{thm:DTpar}, 
the following result holds. 
\begin{thm}
We have the following formula in $\Lambda_{C}$, 
\begin{align}\label{form:thm3}
\mathrm{DT}^{\rm{par}}(\mu, C)
=\prod_{\begin{subarray}{c}
n\in \mathbb{Z}, \gamma \in H_2(C, \mathbb{Z})_{>0}, \\
n/\omega \cdot \gamma =\mu
\end{subarray}}
\exp\left( (-1)^{\gamma \cdot H -1} 
N_{n, \gamma}q^n t^{\gamma}  \right)^{\gamma \cdot H}. 
\end{align}
\end{thm}
\begin{proof}
The same proof of Theorem~\ref{thm:DTpar} is applied. 
\end{proof}

As an analogy of Conjecture~\ref{conj:mult2}
and Conjecture~\ref{conj:prod:exp}, we 
propose the following conjecture. 
\begin{conj}\label{conj:multloc}
For $(n, \gamma) \in \mathbb{Z} \oplus H_2(C, \mathbb{Z})$, 
we have the following formula, 
\begin{align}\label{mult:locform}
N_{n, \gamma}=
\sum_{k\ge 1, k|(n, \gamma)}
\frac{1}{k^2} N_{1, \gamma/k}. 
\end{align}
Or equivalently, we have the following formula
for any $\mu \in \mathbb{Q}$, 
\begin{align*}
\mathrm{DT}^{\rm{par}}(\mu, C)
=\prod_{\begin{subarray}{c}
\gamma \in H_2(C, \mathbb{Z})_{>0}, \\
n/\omega \cdot \gamma=\mu
\end{subarray}}
\left(1-(-1)^{\gamma \cdot H}
q^n t^{\gamma}  \right)^{(\gamma \cdot H)N_{1, \gamma}}.
\end{align*}
\end{conj}

\begin{exam}\label{exam:-12}
Let $f \colon X\to Y$
and $C \subset X$ be as in Example~\ref{exam:-1}, Example~\ref{exam:-1-1}. 
As in Example~\ref{exam:-1}, suppose that 
there is a divisor $H\in \lvert \oO_X(1) \rvert$
which intersects with $C$ at one 
point $p\in C$. Then by Example~\ref{exam:-1-1}, 
we have 
\begin{align*}
\mathrm{DT}^{\rm{par}}(\mu, C)=
\left\{ \begin{array}{cc}
1+q^{\mu}t, & \mu \in \mathbb{Z}, \\
1, & \mbox{ otherwise. } 
\end{array} \right. 
\end{align*}
Hence Conjecture~\ref{conj:multloc}
holds in this case with $N_{1, [C]}=1$
and $N_{1, m[C]}=0$ for $m\ge 2$. 
In particular
for $(m, n) \in \mathbb{Z}^{\oplus 2}$
with $m\ge 1$, 
we have $N_{n, m[C]} \neq 0$
only if $m|n$, and in this case, we have 
\begin{align}\label{Nnm}
N_{n, m[C]}=\frac{1}{m^2}.
\end{align}
\end{exam}
\begin{rmk}
In Example~\ref{exam:-12}, 
there may not exist a divisor $H\subset X$
which intersects with $C$ at a one point
in general. 
However such a divisor always exists on an analytic neighborhood 
of $C \subset U \subset X$, and we can deduce (\ref{Nnm}) 
by the arguments on $U$. 
\end{rmk}
\begin{rmk}
The formula (\ref{Nnm}) is well-known,
c.f.~\cite[Example~6.2]{JS}. 
However the argument in Example~\ref{exam:-12} 
seems to be the easiest 
argument to deduce (\ref{Nnm}). 
\end{rmk}

\subsection{From local theory to global theory}
Finally in this section, we 
show that Conjecture~\ref{conj:multloc} implies Conjecture~\ref{conj:mult2}
using parabolic stable pair invariants. 
\begin{prop}\label{loc:global}
For given $n\in \mathbb{Z}$ and $\beta \in H_2(X, \mathbb{Z})$, 
suppose that the formula (\ref{mult:locform})
holds for any reduced curve $C \stackrel{i}{\hookrightarrow} X$
and $\gamma \in H_2(C, \mathbb{Z})$
with $i_{\ast}\gamma =\beta$. 
Then $N_{n, \beta}$ satisfies the formula (\ref{form:mult:N}). 
\end{prop}
\begin{proof}
We take $d>\omega \cdot \beta$ and 
a divisor $H \subset X$ as in Lemma~\ref{lem:trans}.
We set $\mu=n/\omega \cdot \beta$, 
and consider the invariant 
$\widehat{\DT}_{n, \beta}^{\rm{par}} \in \mathbb{Q}$
as in (\ref{DTtilde}). 
Also we set $\widehat{N}_{n, \beta} \in \mathbb{Q}$ to be
\begin{align}\label{Ntilde}
\widehat{N}_{n, \beta}
\cneq 
\sum_{k\ge 1, k|(n, \beta)} \frac{1}{k^2}
N_{1, \beta/k}
\end{align}
By Remark~\ref{enough:check}, 
it is enough to show the formula, 
\begin{align}\label{wtilde}
\widehat{\DT}^{\rm{par}}_{n, \beta}
=(-1)^{\beta \cdot H -1}(\beta \cdot H)\widehat{N}_{n, \beta}. 
\end{align}
Let $\gamma \in \Chow_{\beta}(X)$ be a one
cycle on $X$ 
and $C_{\gamma} \subset X$ the reduced curve defined by 
$C_{\gamma}=\Supp(\gamma)$. 
By our choice of $H$, the curve $C_{\gamma}$ intersects with $H$
transversally. 
Let us consider the generating 
series $\mathrm{DT}^{\rm{par}}(\mu, C_{\gamma})$ and 
its logarithm. We can similarly define 
$\widehat{\DT}^{\rm{par}}_{n, \gamma} \in \mathbb{Q}$
by the $q^n t^{\gamma}$-coefficient 
of $\log \mathrm{DT}^{\rm{par}}(\mu, C_{\gamma})$.
Also by replacing $\beta$ by $\gamma$ 
in (\ref{Ntilde}), we can similarly define 
$\widehat{N}_{n, \gamma}$.

Let us consider the 
 assignments $\gamma \mapsto 
\widehat{\DT}^{\rm{par}}_{n, \gamma}, \widehat{N}_{n, \gamma}$. 
They determine constructible functions on $\Chow_{\beta}(X)$, 
\begin{align*}
\widehat{\DT}^{\rm{par}}_{n, \ast}
&\colon \Chow_{\beta}(X) \ni \gamma \mapsto 
\widehat{\DT}^{\rm{par}}_{n, \gamma} \in \mathbb{Q}, \\
\widehat{N}_{n, \ast} &\colon \Chow_{\beta}(X) \ni \gamma 
\mapsto \widehat{N}_{n, \gamma} \in \mathbb{Q}. 
\end{align*}
We consider the integrations of these
constructible functions over $\Chow_{\beta}(X)$. 
First note that, by the definition of local parabolic stable pair 
invariant in Definition~\ref{loc:inv},
we have 
\begin{align}\label{int:1}
\mathrm{DT}^{\rm{par}}_{n, \beta}
=\int_{\gamma \in \Chow_{\beta}(X)}
\mathrm{DT}^{\rm{par}}_{n, \gamma}
d \chi. 
\end{align} 
Therefore we have 
\begin{align}\notag
&\int_{\gamma \in \Chow_{\beta}(X)}
\widehat{\DT}^{\rm{par}}_{n, \gamma}
d\chi \\
\notag
&= \int_{\gamma \in \Chow_{\beta}(X)}
\sum_{l\ge 1} \frac{(-1)^{l-1}}{l}
\sum_{\begin{subarray}{c}
(\gamma_i, n_i) \in H_2(C_{\gamma}, \mathbb{Z})\oplus \mathbb{Z}, 
1\le i \le l, \\
(\gamma_1, n_1) + \cdots +(\gamma_l, n_l)=(\gamma, n), \\
n_i=\mu(\omega \cdot \gamma_i)
\end{subarray}}
\prod_{i=1}^{l} \mathrm{DT}_{n_i, \gamma_i}^{\rm{par}} d\chi \\
\notag
&= \sum_{l\ge 1} \frac{(-1)^{l-1}}{l}
\sum_{\begin{subarray}{c}
(\beta_i, n_i) \in H_2(X, \mathbb{Z})\oplus \mathbb{Z}, 
1\le i \le l, \\
(\beta_1, n_1) + \cdots +(\beta_l, n_l)=(\beta, n), \\
n_i=\mu(\omega \cdot \beta_i)
\end{subarray}} 
\int_{\{\gamma_i\}_{i=1}^{l} \in
\prod_{i=1}^{l} \Chow_{\beta_i}(X)}
\prod_{i=1}^{l} \mathrm{DT}_{n_i, \gamma_i}^{\rm{par}} d\chi \\
\notag
&= \sum_{l\ge 1} \frac{(-1)^{l-1}}{l}
\sum_{\begin{subarray}{c}
(\beta_i, n_i) \in H_2(X, \mathbb{Z})\oplus \mathbb{Z},
1\le i\le l, \\
(\beta_1, n_1) + \cdots +(\beta_l, n_l)=(\beta, n), \\
n_i=\mu(\omega \cdot \beta_i)
\end{subarray}} 
\prod_{i=1}^{l} \mathrm{DT}_{n_i, \beta_i}^{\rm{par}} \\
\label{last:eq}
&= \widehat{\DT}^{\rm{par}}_{n, \beta}. 
\end{align}
Here we have used (\ref{int:1}) in the third equation. 

Next, we consider the integration 
of  
the function $\widehat{N}_{n, \ast}$. 
For $a \in \mathbb{Z}_{\ge 1}$, we set 
$\Chow_{\beta}^{(a)}(X) \subset \Chow_{\beta}(X)$ to be 
\begin{align*}
\Chow_{\beta}^{(a)}(X) \cneq 
\{ \gamma \in \Chow_{\beta}(X) : 
\divv(\gamma, n)=a\}, 
\end{align*}
where $\divv(\ast)$ is the divisibility of the vector $\ast$. 
By setting $e=\divv(\beta, n)$, 
we have 
\begin{align}\notag
&\int_{\gamma \in \Chow_{\beta}(X)}
\widehat{N}_{n, \gamma} d\chi \\
\notag
&= \sum_{a \ge 1} \int_{\gamma \in \Chow_{\beta}^{(a)}(X)}
\sum_{k\ge 1, k|a}\frac{1}{k^2}
N_{1, \gamma/k} d\chi \\
\notag
&= \sum_{k\ge 1, k|e}
\frac{1}{k^2}
\sum_{a\ge 1, k|a|e} \int_{\gamma \in \Chow_{\beta}^{(a)}(X)}
N_{1, \gamma/k} d\chi \\
\label{div:k}
&=\sum_{k\ge 1, k|e} \frac{1}{k^2}
\int_{\gamma' \in \Chow_{\beta/k}} N_{1, \gamma'} d\chi \\
\label{eq:4}
&=\sum_{k\ge 1, k|e} \frac{1}{k^2}
N_{1, \beta/k} \\
\label{eq:5}
&= \widehat{N}_{n, \beta}. 
\end{align}
Here (\ref{div:k})
follows from the set theoretic bijection
for $k|e$,  
\begin{align*}
\Chow_{\beta/k}(X) \ni \gamma' \mapsto 
k\gamma' \in \bigcup_{a\ge 1, k|a|e} \Chow_{\beta}^{(a)}(X). 
\end{align*}
Also (\ref{eq:4}) follows from (\ref{N_n}).

Now we use the assumption 
for the local theory. 
It just implies the equality, 
 \begin{align}\label{wtilde2}
\widehat{\DT}^{\rm{par}}_{n, \gamma}
=(-1)^{\gamma \cdot H -1}(\gamma \cdot H)\widehat{N}_{n, \gamma},
\end{align}
for any $\gamma \in \Chow_{\beta}(X)$. 
By (\ref{wtilde2}), (\ref{last:eq}), (\ref{eq:5})
and noting $\gamma \cdot H= \beta \cdot H$, 
 the equality (\ref{wtilde}) follows.  
\end{proof}
As a corollary, we have the following. 
\begin{cor}
Suppose that Conjecture~\ref{conj:multloc} is true for any 
reduced curve $C\subset X$. 
Then Conjecture~\ref{conj:mult2} holds. 
\end{cor}

\begin{rmk}\label{rmk:final}
For $\gamma \in \Chow_{\beta}(X)$, 
let us take the curve $C_{\gamma}=\Supp(\gamma)$
and its normalization, 
\begin{align*}
\widetilde{C}_{\gamma} \to C_{\gamma}. 
\end{align*}
As discussed in~\cite[Lemma~2.11]{Todmu},
the invariant $N_{n, \gamma}$ vanishes 
unless $\widetilde{C}_{\gamma}$
is a disjoint union of $\mathbb{P}^1$. 
The idea for the proof is as follows:
suppose for simplicity that $C_{\gamma}$
is a smooth curve of positive genus, 
 $C_{\gamma} \subset U \subset X$ 
be a sufficiently small analytic 
neighborhood of $C_{\gamma}$ in $X$, 
and $\Pic^0(U)$ the group of line bundles on $U$
which restricts to degree zero line bundles 
on $C_{\gamma}$. The group $\Pic^{0}(U)$ acts 
on the moduli space which defines the invariant 
$N_{n, \gamma}$. 
We can also find a subgroup $S^1 \subset \Pic^0(U)$
whose induced action on the above moduli 
space is free. Hence $N_{n, \gamma}=0$ follows
by the localization argument. 

Now suppose that $\widetilde{C}_{\gamma}$ is 
a union of $\mathbb{P}^1$,  
$C_{\gamma}$ has at worst nodal singularities, and 
the arithmetic genus of $C_{\gamma}$ is positive. 
In this case, the group 
$\Pic^{0}(U)$ contains $\mathbb{C}^{\ast}$, 
so we may try to localize by this action. 
In this case, there may be $\mathbb{C}^{\ast}$-fixed 
sheaves supported on $C_{\gamma}$.  
However it is not easy to investigate the 
contribution of $\epsilon_{0, \gamma, n}$ 
at the $\mathbb{C}^{\ast}$-fixed sheaf, 
and the above localization argument is not 
obvious in this case. 
 
In~\cite{Todmu}, instead of
the invariant $N_{n, \gamma}$, we 
apply the 
$\mathbb{C}^{\ast}$-localization to the parabolic 
stable pair invariant 
$\DT_{n, \gamma}^{\rm{par}}$. 
The moduli space $M_n^{\rm{par}}(C_{\gamma}, \gamma)$
also admits the $\mathbb{C}^{\ast}$-action, 
(while the moduli space of PT or JS stable pairs 
do not,) and the definition of $\mathrm{DT}_{n, \gamma}^{\rm{par}}$ 
does not require the technique on Hall algebras. 
There is no technical difficulty in 
applying the localization on $M_n^{\rm{par}}(X, \gamma)$, 
and can investigate the contribution of 
$\mathbb{C}^{\ast}$-fixed parabolic stable pairs 
to the invariant $\mathrm{DT}^{\rm{par}}_{n, \gamma}$. 
It will turn out that these localization argument is 
relevant to show Conjecture~\ref{conj:multloc} in some cases,
and these details will be pursued in~\cite{Todmu}. 
\end{rmk}

Institute for the Physics and 
Mathematics of the Universe, 

Todai Institute for Advanced Studies (TODIAS), 
University of Tokyo,

5-1-5 Kashiwanoha, Kashiwa, 277-8583, Japan.

\textit{E-mail address}: yukinobu.toda@ipmu.jp

\end{document}